\documentclass[reqno]{llncs}
\usepackage[margin=1.4in]{geometry}
\usepackage[utf8]{inputenc}

\usepackage{amsmath,amsfonts,amssymb}
\usepackage{float}

\usepackage{bm}

\usepackage{lipsum}
\usepackage{ntheorem}
\usepackage{hhline}
\usepackage{soul}

\newcommand{\R}{\mathbb{R}}

\newcommand{\x}{\textbf{x}}

\newtheorem*{theorem-non}{Theorem}
\renewenvironment{proof}{{\bfseries Proof.}}{\hspace*{\fill}$\blacksquare$\vspace{5pt}}

\usepackage{algorithm}
\usepackage{algorithmic}
\usepackage{arydshln,leftidx,mathtools}
\usepackage{blkarray}
\usepackage{tikz}
\usetikzlibrary{arrows}
\usepackage{caption}
\usepackage{subcaption}
\usepackage{tikzsymbols}

\usepackage[colorlinks=true,breaklinks=true,bookmarks=false,urlcolor=blue,citecolor=blue,linkcolor=blue,bookmarksopen=false,draft=false]{hyperref}

\usepackage{todonotes}

\begin{document}

\title{On Combinatorial Network Flows Algorithms and Circuit Augmentation for Pseudoflows\thanks{This work was supported by Air Force Office of Scientific Research grant FA9550-21-1-0233 and NSF grant 2006183, Algorithmic Foundations, Division of Computing and Communication Foundations.}}

\author{Steffen Borgwardt\inst{1}\orcidID{0000-0002-8069-5046} \and Angela Morrison\inst{2}\orcidID{0000-0001-8288-5030}}

\institute{\email{\href{mailto:steffen.borgwardt@ucdenver.edu}{steffen.borgwardt@ucdenver.edu}}; 
University of Colorado Denver
\and \email{\href{mailto:angela.morrison@ucdenver.edu}{angela.morrison@ucdenver.edu}}; University of Colorado Denver
}

\date{}

\maketitle

\begin{abstract}
There is a wealth of combinatorial algorithms for classical min-cost flow problems and their simpler variants like max flow or shortest path problems. It is well-known that many of these algorithms are related to the Simplex method and the more general circuit augmentation schemes: prime examples are the network Simplex method, a refinement of the primal Simplex method, and min-mean cycle canceling, which corresponds to a steepest-descent circuit augmentation scheme. 

We are interested in a deeper understanding of the relationship between circuit augmentation and combinatorial network flows algorithms. To this end, we generalize from the consideration of primal or dual flows to so-called pseudoflows, which adhere to arc capacities but allow for a violation of flow balance. We introduce `pseudoflow polyhedra,' in which slack variables are used to quantify this violation, and characterize their circuits. This enables the study of combinatorial network flows algorithms in view of the walks that they trace in these polyhedra, and in view of the pivot rules for the steps. In doing so, we provide an `umbrella,' a general framework, that captures several algorithms. 

We show that the Successive Shortest Path Algorithm for min-cost flow problems, the Shortest Augmenting Path Algorithm for max flow problems, and the Preflow-Push algorithm for max flow problems lead to (non-edge) circuit walks in these polyhedra. The former two are replicated by circuit augmentation schemes for simple pivot rules. Further, we show that the Hungarian Method leads to an edge walk and is replicated, equivalently, as a circuit augmentation scheme or a primal Simplex run for a simple pivot rule. 

\end{abstract}

\section{Introduction}
The studies of the connection between combinatorial algorithms and linear programming over the underlying polyhedra are a classical field in optimization theory. We are interested in a linear-programming interpretation of combinatorial algorithms for min-cost flow problems, which encompass classical transportation problems, max flow, or shortest path problems. Such connections lead to a possible transfer of knowledge, simpler analysis or implementations, or valuable information for the design of new algorithms. 

Formally, a general min-cost flow problem on a network $G=(N,A)$ with $n$ nodes and $m$ arcs can be represented as a linear program (LP)
\begin{equation*}
\begin{array}{rcrcrcr}
    \text{min }  && &&\bm{c}^T\bm{x} & &   \\
    \text{s. t. } && &&A_G\bm{x} &=& \bm{b} \;\\
    && && \bm{\ell} \leq \bm{x} &\leq & \bm{u}, \tag{MCF}
\end{array}\label{MCF}
\end{equation*}
where $\bm{c}\in \mathbb{R}^m$ is the cost of each arc in the network, $A_G \in \mathbb{R}^{n\times m}$ is the node-arc incidence matrix, $\bm{b}\in \mathbb{R}^n$ is the supply/demand value for each node, and $\bm{\ell}\in \mathbb{R}^m$ and $\bm{u}\in \mathbb{R}^m$ are the lower and upper capacities of each arc, respectively. The equality constraints are referred to as {\em flow balance constraints}, guaranteeing that the specified balance between in-flow and out-flow is satisfied for each node. We use the tag (\ref{MCF}) to refer to the polyhedron underlying the LP.

Many combinatorial network flows algorithms allow an interpretation in terms of primal and primal-dual Simplex methods over (\ref{MCF}). Well-known examples include the network Simplex method, a tailored variant of the primal Simplex method \cite{d-51,o-97}, the Hungarian method for assignment problems \cite{h-55,m-57,m-06}, and the more general primal-dual algorithm, literally named for its use of primal and dual Simplex information \cite{amo-93,ff-57,gw-96}.

A generalization from primal Simplex methods to the more general class of {\em circuit augmentation schemes} leads to another prototypical example: a cycle-canceling scheme corresponds to a {\em circuit walk} over (\ref{MCF}). In particular, min-mean cycle canceling is an example of a steepest-descent circuit augmentation scheme \cite{gdl-14}. A standard reformulation of a max flow problem as a circulation problem allows an interpretation of the Edmonds-Karp-Dinic algorithm as a cycle canceling scheme, too.

In this work, we provide interpretations for three combinatorial network flows algorithms as circuit augmentation schemes for so-called {\em pseudoflows}, which allow for a violation of flow balance constraints, but must adhere to the capacity constraints. Notably, these interpretations will be purely primal; no dual information is used. In Section \ref{sec:def_pseudo_poly}, we formulate a {\em pseudoflow polyhedron} (\ref{pseudo_poly_form}) and characterize its set of circuits. The formulation of (\ref{pseudo_poly_form}) is based on tracking the violation of flow balance constraints through a set of slack variables; (\ref{MCF}) appears as a face of the polyhedron, and a zero pseudoflow forms a vertex. This enables the connection of combinatorial algorithms that gradually build an optimal flow from a zero start to circuit augmentation schemes over (\ref{pseudo_poly_form}), in Section \ref{sec:alg_of_int}.

\subsection{Preliminaries}\label{sec:preliminaries}

Before we describe our contributions, we provide some background from the theory of circuits and circuit augmentation. We follow \cite{bfh-14,bv-17,dhl-15,env-22,r-69}. Recall  that a set of \textit{circuits} for a polyhedron can be defined as follows. 

\begin{definition}[Circuits]\label{def:circuits}
For a polyhedron $P = \{\bm{x}\in \mathbb{R}^n: A\bm{x} = \bm{b}, B\bm{x}\leq \bm{d}\}$, the set of circuits of $P$, denoted $\mathcal{C}(A,B)$, consists of those $\bm{g} \in ker(A)\setminus \{0\}$, normalized to coprime integer components, for which $B\bm{g}$ is support minimal over the set of $\{B\bm{x}:\bm{x}\in ker(A) \setminus \{0\}\}$.\\

\end{definition}
The set of circuits has been shown to consist of all potential edge directions of the polyhedron as the right hand sides $\bm{b}$ and $\bm{d}$ vary \cite{g-75}. Therefore, it contains the set of all actual edge directions.

Using the set of circuits as the possible directions for steps of a walk over a polyhedron leads to the concept of a \textit{circuit walk}:
 \begin{definition}[Circuit Walk]\label{def:circuitwalk}
 Let $P = \{\bm{x}\in \mathbb{R}^n: A\bm{x} = \bm{b}, B\bm{x}\leq \bm{d}\}$ be a polyhedron. For a vertex $\bm{v}$ of $P$, we call a sequence of $\bm{v} = \bm{x_0},...,\bm{x_k}$ a circuit walk of length $k$ if for $i = 0,...,k-1$:
 \begin{enumerate}
     \item $\bm{x_i} \in P$
     \item $\bm{x_{i+1}} = \bm{x_i} +\alpha_i \bm{g_i}$ for some $\bm{g_i} \in \mathcal{C}(A,B)$ and $\alpha_i > 0$ and
     \item $\bm{x_i}+\alpha_i \bm{g_i}$ is infeasible for all $\alpha > \alpha_i$.
 \end{enumerate}
 \end{definition}
A circuit $\bm{g} \in \mathcal{C}(A,B)$ such that $\bm{x_i} + \alpha \bm{g} \in P$ for some $\alpha > 0$ is called a {\em feasible circuit} at $\bm{x_i}\in P$, or simply feasible, if $\bm{x_i}$ is clear from the context. Our definition assumes that feasible circuits are used for steps of maximal length, and that the walk begins at a vertex $\bm{v} = \bm{x_0}$. We do not require that the walk terminates at a vertex $\bm{x_k}$.

In \cite{bdf-16}, a hierarchy of different types of circuit walks was studied. We summarize this hierarchy into three types of circuit walks: edge walks, vertex walks, and general walks. 

\textit{Edge walks}, see Figure \ref{fig:edge_walk}, are walks through the polyhedron where each step is along an edge of the polyhedron. This is the type of walk performed by the Simplex method. A vertex circuit walk, or simply \textit{vertex walk}, shown in Figure \ref{fig:non-edge_walk}, is a walk for which all $\bm{x_i}$ are vertices, and at least one step of the walk is not along an edge. A general circuit walk, or simply \textit{general walk}, is neither an edge walk nor a vertex walk; see Figure \ref{fig:gen_walk} for an example. A general walk has at least one $\bm{x_i}$ that is not a vertex. For the purposes of this paper, we are interested in edge and general circuit walks. 

\begin{figure}[b]
     \centering
     \begin{subfigure}[b]{0.3\textwidth}
         \centering
         \subcaptionbox{Edge Walk \label{fig:edge_walk}}{
\begin{tikzpicture}[scale = 0.75]
    \foreach \x in {0,1,...,4} {
        \foreach \y in {0,1,...,4} {
            \fill[color=black] (\x,\y) circle (0.05);
        }
    }
\draw[-,ultra thick] (1,1)--(3,1);
\draw[-,ultra thick] (3,1)--(3,3);
\draw[-,ultra thick] (3,3)--(1,3);
\draw[-,ultra thick] (1,3)--(1,1);
\draw[-latex,ultra thick, red] (1,1)--(3,1);
\draw[-latex,ultra thick, red] (3,1)--(3,3);
\draw[-latex,ultra thick, red] (3,3)--(1,3);
\end{tikzpicture}}
     \end{subfigure}
     \hfill
     \begin{subfigure}[b]{0.3\textwidth}
         \centering
         \subcaptionbox{Vertex Walk \label{fig:non-edge_walk}}{
\begin{tikzpicture}[scale = 0.5]
    \foreach \x in {0,1,...,6} {
        \foreach \y in {0,1,...,6} {
            \fill[color=black] (\x,\y) circle (0.05);
        }
    }
\draw[-,ultra thick] (2,0)--(4,0);
\draw[-,ultra thick] (4,0)--(6,2);
\draw[-,ultra thick] (6,2)--(6,4);
\draw[-,ultra thick] (6,4)--(4,6);
\draw[-,ultra thick] (4,6)--(2,6);
\draw[-,ultra thick] (2,6)--(0,4);
\draw[-,ultra thick] (0,4)--(0,2);
\draw[-,ultra thick] (0,2)--(2,0);
\draw[-latex,ultra thick, red] (0,2)--(6,2);
\draw[-latex,ultra thick, red] (6,2)--(2,6);
\draw[-latex,ultra thick, red] (2,6)--(4,6);
\draw[-latex,ultra thick, red] (4,6)--(4,0);
\end{tikzpicture}}
     \end{subfigure}
     \hfill
     \begin{subfigure}[b]{0.3\textwidth}
         \centering
         \subcaptionbox{General Walk \label{fig:gen_walk}}{
\begin{tikzpicture}[scale = 0.75]
    \foreach \x in {0,1,...,4} {
        \foreach \y in {0,1,...,4} {
            \fill[color=black] (\x,\y) circle (0.05);
        }
    }
\draw[-,ultra thick] (1,4)--(0,3);
\draw[-,ultra thick] (0,3)--(0,2);
\draw[-,ultra thick] (0,2)--(2,0);
\draw[-,ultra thick] (2,0)--(4,2);
\draw[-,ultra thick] (4,2)--(2,4);
\draw[-,ultra thick] (2,4)--(1,4);
\draw[-latex,ultra thick, red] (1,4)--(3.5,1.5);
\draw[-latex,ultra thick, red] (3.5,1.5)--(2,0);
\end{tikzpicture}}
     \end{subfigure}
        \caption{Types of Circuit Walks \cite{bv-17}.}
        \label{fig:type_of_circ_walks}
\end{figure}
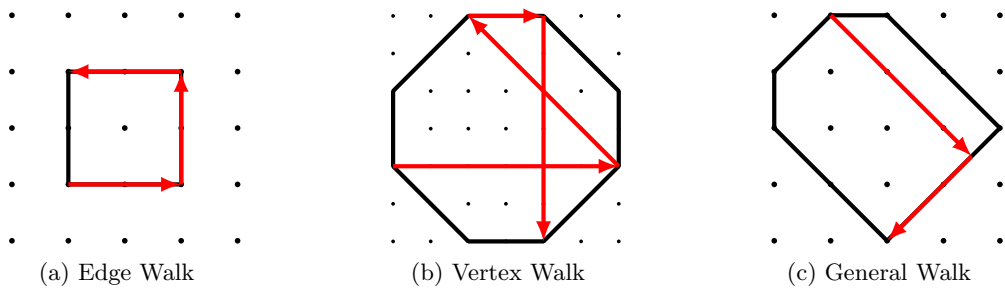

Circuit walks are relevant in the studies of circuit diameters \cite{bfh-14} and for \textit{circuit augmentation schemes} \cite{dhl-15}. Circuit augmentation can be used to solve linear programs \cite{bv-19b,dhl-15,gdl-14}; it is a direct generalization of the primal Simplex method. As the set of circuits can be used to define optimality criteria for a linear program \cite{g-75}, there always exists an improving circuit until one reaches an optimum. In practice, circuit augmentation is useful for traversing highly degenerate polyhedra \cite{gdl-14}. 

We use the Simplex terminology of a pivot and a pivot rule to discuss the choice of improving circuit in each step. Our goal is not only to exhibit that certain combinatorial algorithms lead to circuit walks over a pseudoflow polyhedron, but also to identify the pivot rule for a circuit augmentation scheme that replicates the walk. The algorithms discussed in Section \ref{sec:alg_of_int} relate to {\em Dantzig's rule} and {\em steepest-ascent}. We recall formal definitions for a maximization problem \cite{dhl-15}. 

\begin{definition}[Dantzig's Rule]\label{def:dantzig}
Given a polyhedron $P = \{\bm{x}\in \mathbb{R}^n: A\bm{x} = \bm{b}, B\bm{x}\leq \bm{d}\}$, a feasible solution $\bm{x_0} \in P$, and an objective function $\bm{c} \in \mathbb{R}^n$, choose the circuit $\bm{g} \in \mathcal{C}(A,B)$ for which $\bm{c^T}\bm{g}$ is maximized among all feasible circuits at $\bm{x_0}$.
\end{definition}

Dantzig's rule can be seen as a generalization or relaxation of  a largest-coefficient rule in the Simplex method \cite{v-14}. A {\em Dantzig circuit} $\bm{g}$ provides the overall best improvement $\bm{c^T}\bm{g}$ to the objective function value if added to the current solution $\bm{x_0}$. For a minimization problem, the value of $-\bm{c^T}\bm{g}$ is maximized. In contrast, steepest-ascent takes into account a normalization of the length of the circuit \cite{bv-19b}. 

\begin{definition}[Steepest-ascent Circuit]\label{def:steepest}
Given a polyhedron $P = \{\bm{x}\in \mathbb{R}^n: A\bm{x} = \bm{b}, B\bm{x}\leq \bm{d}\}$, a feasible solution $\bm{x_0} \in P$, and an objective function $\bm{c} \in \mathbb{R}^n$, a steepest-ascent circuit at $\bm{x_0}$ is a circuit $\bm{g} \in \mathcal{C}(A,B)$ that maximizes $\bm{c^T}\bm{g}/||B\bm{g}||_1$ among all feasible circuits at $\bm{x_0}$.
\end{definition}
Here, $||B\bm{g}||_1$ denotes the $1$-norm of vector $B\bm{g}$. A {\em steepest-descent circuit} is defined analogously, with minimization instead of maximization. 

\subsection{Contributions and Outline} 

Our main contribution is to show that three classical network flows algorithms are combinatorial interpretations of circuit augmentation over {\em pseudoflow polyhedra}. We introduce and study these polyhedra in Section \ref{sec:def_pseudo_poly}. In Section \ref{sec:alg_of_int}, we exhibit the types of circuit walks constructed in a run of the combinatorial algorithms, and devise objective functions and pivot rules that lead to the walks. We conclude with a brief outlook and a fourth algorithm that traces a circuit walk in Section \ref{sec:conclusion}. 

In Section \ref{sec:formaldef}, we devise a formal definition of a pseudoflow polyhedron (\ref{pseudo_poly_form}) (Definition \ref{def:pseudoflow}); the formulation is based on an explicit representation of the possible violation of flow balance through a pair of slack variables for each node. In Section \ref{sec:vertices}, we prove the existence of a vertex in (\ref{pseudo_poly_form}) corresponding to a zero pseudoflow (Lemma \ref{lem:feas_vert}); it forms the start of the circuit walks constructed in Section \ref{sec:alg_of_int}.

In Section \ref{sec:circuits}, we provide a characterization of the set of circuits of (\ref{pseudo_poly_form}). We show that they correspond to simple undirected cycles underlying a {\em pseudoflow network} (Theorem \ref{thm:pseudo_circs}). We translate this characterization to the description of (\ref{pseudo_poly_form}) (Corollary \ref{cor:final_circ_char}), distinguishing different types of circuits depending on their use of  pseudoflow variables and slack variables. This exhibits the possible steps of circuit augmentation schemes over (\ref{pseudo_poly_form}): in each iteration, the current pseudoflow is changed by sending flow along a path (a {\em path circuit)} or along a cycle (a {\em cycle circuit}). 
 
In Section \ref{sec:alg_of_int}, we study three combinatorial network flows algorithms that are based on iterative updates to pseudoflows. In Section \ref{sec:sspa}, we begin with the Successive Shortest Path Algorithm (Algorithm \ref{alg:SSPA}) for general min-cost flow problems \cite{e-82}.
We prove that it performs a circuit walk over (\ref{pseudo_poly_form}) (Lemma \ref{lem:sspa_circ_walk}) and exhibit that it performs a general walk; in fact, a flow at termination of the Successive Shortest Path Algorithm is not necessarily a vertex of (\ref{MCF}). We prove that a circuit augmentation scheme using Dantzig's rule replicates a run of the algorithm (Theorem \ref{thm:sspa_dantzig}).

In Section \ref{sec:GAPA}, we extend these results to the Generic Augmenting Path Algorithm (Algorithm \ref{GAP}) for max flow problems. As we will see, the algorithm essentially is a special case of the Successive Shortest Path Algorithm; the pseudoflow network in this simpler setting captures the standard reformulation of max flow as a circulation problem. We prove that the Generic Augmenting Path Algorithm also performs a circuit walk over (\ref{pseudo_poly_form}) (Lemma \ref{lem:gapa_circ_walk}) and again exhibit that it performs a general walk. We prove that any pivot rule replicates a run of the algorithm (Theorem \ref{thm:gap_aug}); if shortest paths are used for the augmentation, one can provide an objective function such that a steepest-ascent circuit augmentation scheme replicates a run (Corollary \ref{cor:sapa_saa}).

Next, we study the Hungarian Method (Algorithm \ref{alg:hung}) for assignment problems. We prove that the Hungarian Method performs a circuit walk over (\ref{pseudo_poly_form}) (Lemma \ref{lem:hm_circ_walk}). Further, we show that for assignment problems all feasible circuits at vertices of a certain face of (\ref{pseudo_poly_form}), in fact, follow an edge direction (Theorem \ref{thm:hm_edge_walk}). This implies that the Hungarian Method performs an edge walk over (\ref{pseudo_poly_form}) (Corollary \ref{cor:hm_edge_walk}). We prove that a circuit augmentation scheme using Dantzig's rule over (\ref{pseudo_poly_form}) replicates a run of the algorithm (Theorem \ref{thm:hm_dantzig}), as does a primal Simplex run with the largest-coefficient rule (Corollary \ref{cor:hm_dantzig}). We arrive at purely primal interpretations of these algorithms, whereas the Successive Shortest Path Algorithm and the Hungarian Method are typically understood to be primal-dual algorithms. Essentially, we identify that the framework of (primal) circuit augmentation over the pseudoflow polyhedron is an `umbrella' able to capture these algorithms. 

Finally, in Section \ref{sec:conclusion}, we explain some promising directions for further research. In particular, we show some preliminary results for the Preflow-Push Algorithm (Algorithm \ref{alg:PFP}). Going through similar motions as in Section \ref{sec:alg_of_int}, we prove that the Preflow-Push Algorithm performs a circuit walk (Lemma \ref{lem:pfp_circ_walk}) and provide an example in which it performs a general walk. The existence of a pivot rule that can replicate a run of the algorithm remains an open question.

\section{Pseudoflow Polyhedron}\label{sec:def_pseudo_poly}
We begin with a formal definition of a \textit{pseudoflow polyhedron}. Then, we exhibit the existence of a special vertex corresponding to a zero pseudoflow and provide a characterization of the set of circuits of the polyhedron. These observations will be valuable tools for our discussion in Section \ref{sec:alg_of_int}.

\subsection{Formal Definition}\label{sec:formaldef}

The main constraints of min-cost flow problems are the so-called {\em flow balance constraints}. In (\ref{MCF}), they form the equality system $A_G\bm{x}=\bm{b}$. The flow balance constraints force nodes $i$ to satisfy certain supply ($b_i>0$) or demand  ($b_i<0$) values, measured as the difference of outgoing and incoming flow. Algebraically, the flow balance constraint for node $i$ is $$\sum\limits_{(i,j)\in A} x_{ij} - \sum\limits_{(j,i)\in A} x_{ji} = b_i.$$
In this paper, we are interested in algorithms for min-cost flow problems that build on pseudoflows instead of primal feasible flows. Formally, a \textit{pseudoflow} is a function $\bm{x}:A\rightarrow \mathbb{R}^+$  which satisfies the capacity constraints, i.e., the domain constraints $\bm{\ell}\leq \bm{x}\leq \bm{u}$ of (\ref{MCF}), but not necessarily the flow balance constraints \cite{o-97}. In the literature, pseudoflows also appear as capacity-feasible flows.

It is easy to translate any min-cost flow problem into an equivalent one with capacity constraints $\bm{0_m} \leq\bm{x} \leq \bm{u'}$, where $\bm{0_m}$ is a zero vector of length $m \in \mathbb{Z}^+$. For the sake of a simple discussion, we assume $\bm{\ell}=\bm{0_m}$. Therefore a minimal representation of the set of pseudoflows in a given network would be 
\begin{equation*}
\begin{array}{rccccr}
    &\bm{0_m}&\leq& \bm{x} &\leq& \bm{u}.\\
\end{array}
\end{equation*}

Our goal in Section \ref{sec:alg_of_int} is an interpretation of combinatorial network flows algorithms as circuit augmentation schemes over the set of pseudoflows. To this end, we need an explicit representation of a violation of flow balance. This violation will decrease (or increase) throughout an algorithm until one arrives at a primal feasible and optimal flow. This concept is similar to a primal-dual approach to (\ref{MCF}), but its implementation contrasts sharply: we are working towards a purely {\em primal} interpretation of algorithms without use of dual variables. In contrast, restricted primal problems in a primal-dual approach minimize violation of complementary slackness with respect to a current dual feasible solution \cite{gw-96}.

We introduce {\em relaxed flow balance constraints}, where non-negative slack variables $s^+_{i}$ and $s^-_{i}$ for each node $i$ `correct' a violation of flow balance by setting
$$\sum\limits_{(i,j)\in A} x_{ij} - \sum\limits_{(j,i)\in A} x_{ji} -s^+_{i}+s^-_{i} = b_i.$$ The slack variables $s^+_{i}$ and $s^-_{i}$ allow a measurement of the violation of flow balance at node $i$, as well as a distinction between a positive or negative violation. We provide a formal definition.

\begin{definition}[Pseudoflow Polyhedron]\label{def:pseudoflow}
Let $G=(N, A)$ be a directed network, $A_G \in \mathbb{R}^{n \times m}$ the node-arc incidence matrix for $G$, $\bm{b} \in \mathbb{R}^{n}$ the supply or demand values for each node, and $\bm{u}\in \mathbb{R}^{m}$ the upper capacities for each arc. The {\em pseudoflow polyhedron} for $G$ is defined as the set of $(\bm{x},\bm{s^+},\bm{s^-})$ satisfying 
\begin{equation*}
\begin{array}{rcrcrcr}
    A_G\bm{x} & - & I\bm{s^+} & + & I\bm{s^-}   & =   & \bm{b}\\
         &   &      & \bm{0_m} & \leq \bm{x} & \leq & \bm{u}\\
         &   &      &   & \bm{s^+}      & \geq & \bm{0_n}\\ 
         &   &      &   & \bm{s^-}      & \geq & \bm{0_n} \tag{pseudo}
\end{array}\label{pseudo_poly_form}
\end{equation*}

Here, $\bm{x} \in \mathbb{R}^{m}$ corresponds to a pseudoflow on the arcs of the network, and $\bm{s^+},\bm{s^-} \in \mathbb{R}^{n}$ are slack variables for each node.
\end{definition}

For a simple wording, we use the tag (\ref{pseudo_poly_form}) to refer both to the formulation itself and to the underlying polyhedron, depending on the context.

\subsection{Vertices}\label{sec:vertices}

Next, we take a closer look at a vertex of (\ref{pseudo_poly_form}) that will serve as a starting solution for the algorithms discussed in Section \ref{sec:alg_of_int}. 

Observe that (\ref{pseudo_poly_form}) lies in $\mathbb{R}^{m+2n}$. In order for a point $\bm{y}=(\bm{x},\bm{s^+},\bm{s^-})$ to be a vertex of (\ref{pseudo_poly_form}) there must be $m+2n$ linearly independent active constraints (i.e., constraints satisfied with equality) at $\bm{y}$. Thus, if one can show that the rank of a submatrix of row vectors of active constraints at some $\bm{y}$ in (\ref{pseudo_poly_form}) is at least $m+2n$, then $\bm{y}$ is a vertex. We prove that a vector $(\bm{x},\bm{s^+},\bm{s^-}) \in \mathbb{R}^{m+2n}$ with $\bm{x}=\bm{0_m}$ and $\bm{s^+},\bm{s^-}$ chosen suitably is a vertex of (\ref{pseudo_poly_form}), and that this vertex is unique. 

\begin{lemma}\label{lem:feas_vert}
A vector $(\bm{x},\bm{s^+},\bm{s^-}) \in \mathbb{R}^{m+2n}$ where $\bm{x}=\bm{0_m}$ and, for all $i\in N$, $s^+_i = -b_i$, $s^-_i = 0$ for $b_i \leq 0$ and $s^-_i=b_i$, $s^+_i=0$ for $b_i > 0$ is a vertex of the pseudoflow polyhedron. Further, it is the unique vertex with $\bm{x}=\bm{0_m}$.
\end{lemma}

\begin{proof}
The $n$ relaxed flow balance constraints are equality constraints, and thus always active. The corresponding row vectors are linearly independent: $\bm{s^+}$ and $\bm{s^-}$ appear in them in combination with an identity matrix. It remains to identify which domain constraints from $\bm{x}\geq \bm{0_m}$, $\bm{s^\pm} \geq \bm{0_n}$ can be added to the set of equality constraints while retaining linear independence. 

As $\bm{x}=\bm{0_m}$, the corresponding constraints $\bm{x}\geq \bm{0_m}$ are active and the constraints $\bm{x}\leq \bm{u}$ are inactive. By construction, for each pair $s^\pm_i$, at least one constraint is active: $s^-_i = 0$ for $b_i \leq 0$ and $s^+_i=0$ for $b_i > 0$. We call these $n$ variables `guaranteed to be zero'. 

Consider a matrix description of (\ref{pseudo_poly_form}) and begin with an ordering of rows and columns as in the formulation. Reorder the columns such that the $n$ variables $s_i^\pm$ not guaranteed to be zero appear first. Reorder the rows such that the constraints $\bm{x}\geq \bm{0_m}$ appear below the equality constraints, followed by the constraints $s_i^\pm \geq 0$ for the $s_i^\pm$ guaranteed to be zero. 
 
Now, row-reduce the submatrix of equality constraints to reduced echelon form. The result for the whole constraint matrix is an upper triangular form with all $m+2n$ entries on the main diagonal equal to $1$. This implies that the $m$ constraints $\bm{x}\geq \bm{0_m}$ and the $n$ constraints $s_i^\pm \geq 0$ for $s_i^\pm$ guaranteed to be zero can be added to the set of active constraints while retaining linear independence. This gives a total of $n+m+n=m+2n$ linearly independent rows.

Uniqueness follows from the fact that vertices are necessarily formed from an inclusion-maximal set of active constraints. The choice of $s^\pm_i$ always sets at least one of the variables to $0$, and both if $b_i=0$. This proves the claim.
\end{proof}

Recall that circuit augmentation schemes are defined to start at a vertex (Definition \ref{def:circuitwalk}). Lemma \ref{lem:feas_vert} shows that a zero pseudoflow corresponds to a unique vertex of (\ref{pseudo_poly_form}). A combinatorial algorithm starting with a zero pseudoflow can be interpreted to start at that vertex.

\subsection{Circuit Characterization}\label{sec:circuits}

We now turn to a characterization of the circuits of (\ref{pseudo_poly_form}). First, observe that the constraint matrix of (\ref{pseudo_poly_form}) is totally-unimodular. Node-arc incidence matrices, such as $A_G$, are instances of a totally-unimodular matrix. The matrix of (\ref{pseudo_poly_form}) combines $A_G$ with identity matrices for the slack variables to form the equality constraints, and the inequality constraints are (positive or negative) unit rows for all variables. Thus, the combined constraint matrix retains total unimodularity. This implies that the circuits of (\ref{pseudo_poly_form}) can only contain $\pm 1$ or $0$ as entries \cite{o-10}. A value of $-1$ for an arc is interpreted as flow in the opposite direction of the arc.

In fact, the relation to node-arc incidence matrices leads to a complete answer. It is well known that the circuits of a node-arc incidence matrix, or equivalently (\ref{MCF}), are the (simple) cycles of the underlying undirected network, with flow $\pm 1$ on each arc determined by an arbitrary orientation of flow along the cycle; see, e.g., \cite{amo-93,o-06}. We transfer this characterization to (\ref{pseudo_poly_form}) by exhibiting that it represents the set of feasible flows (\ref{MCF}) in a transformed network: one can view the slack variables in (\ref{pseudo_poly_form}) as \textit{slack arcs} which connect each node in the network to a dummy node outside of the original network. Specifically, $s^+_i$ is an arc from the dummy node to node $i$, while $s^-_i$ is an arc from node $i$ to the dummy node. The addition of the dummy node and slack arcs yields a new network, which we call the \textit{pseudoflow network}; an example is shown in Figure \ref{fig:trans_net}. 

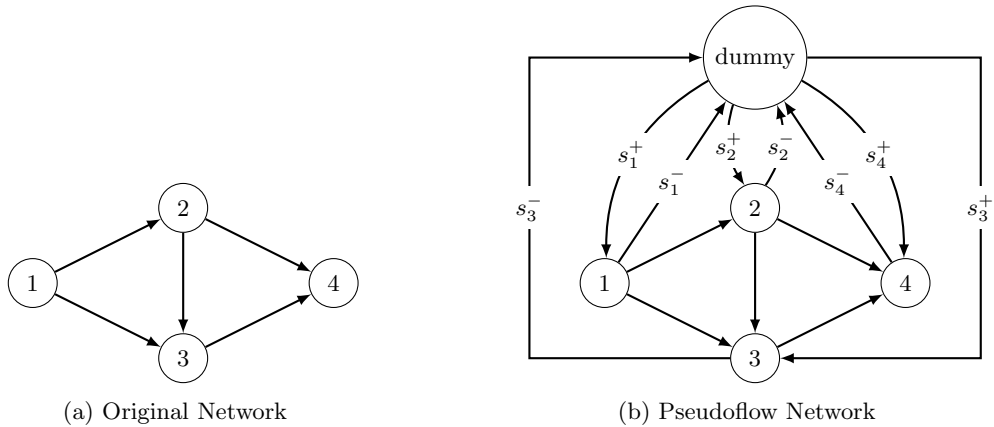
\begin{figure}[t]
     \centering
     \begin{subfigure}[b]{0.475\textwidth}
         \centering
         \subcaptionbox{Original Network \label{fig:orig_net}}{
\begin{tikzpicture}[scale = 0.5]
\tikzset{vertex/.style = {shape=circle,draw,minimum size=2em}}
\tikzset{edge/.style = {-triangle 90,fill=black}}
\tikzset{edgedd/.style = {dashed,-triangle 90,fill=black}}
\node[vertex] (1) at (-4,0) {1};
\node[vertex] (2) at (0,2) {2};
\node[vertex] (3) at (0,-2) {3};
\node[vertex] (4) at (4,0) {4};
\draw[-latex,thick] (1)--(2) {};
\draw[-latex,thick] (1)--(3) {};
\draw[-latex,thick] (2)--(3) {};
\draw[-latex,thick] (2)--(4) {};
\draw[-latex,thick] (3)--(4) {};
\end{tikzpicture}}
     \end{subfigure}
     \hfill
     \begin{subfigure}[b]{0.475\textwidth}
         \centering
         \subcaptionbox{Pseudoflow Network \label{fig:pseudo_net}}{
\begin{tikzpicture}[scale = 0.5]
\tikzset{vertex/.style = {shape=circle,draw,minimum size=2em}}
\node[vertex] (dummy) at (0,6) {dummy};
\node[vertex] (1) at (-4,0) {1};
\node[vertex] (2) at (0,2) {2};
\node[vertex] (3) at (0,-2) {3};
\node[vertex] (4) at (4,0) {4};
\draw[-latex,thick] (1)--(2) {};
\draw[-latex,thick] (1)--(3) {};
\draw[-latex,thick] (2)--(3) {};
\draw[-latex,thick] (2)--(4) {};
\draw[-latex,thick] (3)--(4) {};
\draw[-latex,thick] (dummy) to [bend right = 25] node[midway, fill=white] {$s^+_{2}$} (2);
\draw[-latex,thick] (2) to [bend right = 25] node[midway, fill=white] {$s^-_{2}$} (dummy);
\draw[-latex,thick] (dummy) to [bend right = 30] node[midway, fill=white] {$s^+_{1}$} (1);
\draw[-latex,thick] (1) to [bend right = 0] node[midway, fill=white] {$s^-_{1}$} (dummy);
\draw[-latex,thick] (dummy) to [bend right = -30] node[midway, fill=white] {$s^+_{4}$} (4);
\draw[-latex,thick] (4) to [bend right = 0] node[midway, fill=white] {$s^-_{4}$} (dummy);
\draw[-latex, thick] (3)--(-6,-2) -- node[midway, fill=white] {$s^-_3$} (-6,6) -- (dummy);
\draw[-latex, thick] (dummy)--(6,6) -- node[midway, fill=white] {$s^+_3$} (6,-2) -- (3);
\end{tikzpicture}}
     \end{subfigure}
        \caption{An example of a network (left) that is transformed into its corresponding pseudoflow network (right).}
        \label{fig:trans_net}
\end{figure}

We formally verify these observations and obtain a characterization of the circuits of the pseudoflow polyhedron in terms of the related pseudoflow network.

\begin{theorem}\label{thm:pseudo_circs}
The circuits of a pseudoflow polyhedron $P$ are the simple undirected cycles underlying the pseudoflow network, $G_P$, with flow $\pm 1$ on each arc determined by an arbitrary orientation of flow along the cycle.
\end{theorem}

\begin{proof}
It suffices to prove that the kernel of a node-arc incidence matrix $A_P$ for the pseudoflow network $G_P$ is equal to the kernel of the matrix $F = \begin{bmatrix} A_G & -I & I \end{bmatrix}$ of the flow balance constraints in (\ref{pseudo_poly_form}). 

We construct $A_P$. $F$ corresponds to the rows of $A_P$ for all nodes $i$ in the original network $G$. We define $\bm{d} = (\bm{0_m}, \bm{1_n}, \bm{-1_n})$ to be the flow balance row vector for the dummy node in the pseudoflow network. Here, $\bm{1_n}$ denotes a row vector of $n$ consecutive ones. Then 
$$A_P = \begin{bmatrix} F \\ \bm{d}\end{bmatrix}.$$

The rows of any node-arc incidence matrix, such as $A_P$, are linearly dependent: each column has precisely two non-zero entries, a 1 and a -1. This implies that $\bm{d}$ can be expressed as a linear combination of the rows in $F$. In other words, $F$ is a reduced node-arc incidence matrix for $G_P$. Thus $\text{ker }A_P = \text{ker }F$, which proves the claim.
\end{proof}

Theorem \ref{thm:pseudo_circs} provides a characterization of the set of circuits of (\ref{pseudo_poly_form}) with respect to the pseudoflow network. We now describe and categorize the circuits utilizing knowledge of the original network. We denote an undirected path from node $i$ to node $j$ (and oriented from $i$ to $j$) in the original network as $H_{ij}$ , and an undirected cycle (with some orientation) in the original network as $C$. For convenience, we also use $\bm{H_{ij}} \in \mathbb{R}^{m}$ and $\bm{C} \in \mathbb{R}^{m}$ to denote the vector with flow $\pm 1$ on the corresponding arcs with respect to the orientation. We use $\bm{e_i} \in \mathbb{R}^{n}$ to denote the $i^{th}$ unit vector.

\begin{figure}[b!]
     \centering
     \begin{subfigure}[b]{0.475\textwidth}
         \centering
         \subcaptionbox{Trivial Circuit (dashed) \label{fig:triv_circ}}{
\begin{tikzpicture}[scale = 0.5]
\tikzset{vertex/.style = {shape=circle,draw,minimum size=2em}}
\tikzset{edge/.style = {-triangle 90,fill=black}}
\tikzset{edgedd/.style = {dashed,-triangle 90,fill=black}}
\node[vertex] (dummy) at (0,6) {dummy};
\node[vertex] (1) at (-4,0) {1};
\node[vertex] (2) at (0,2) {2};
\node[vertex] (3) at (0,-2) {3};
\node[vertex] (4) at (4,0) {4};
\draw[-latex,thick] (1)--(2) {};
\draw[-latex,thick] (1)--(3) {};
\draw[-latex,thick] (2)--(3) {};
\draw[-latex,thick] (2)--(4) {};
\draw[-latex,thick] (3)--(4) {};
\draw[-latex,thick] (dummy) to [bend right = 25] node[midway, fill=white] {$s^+_{2}$} (2);
\draw[-latex,thick] (2) to [bend right = 25] node[midway, fill=white] {$s^-_{2}$} (dummy);
\draw[-latex,thick] (dummy) to [bend right = 30] node[midway, fill=white] {$s^+_{1}$} (1);
\draw[-latex,thick] (1) to [bend right = 0] node[midway, fill=white] {$s^-_{1}$} (dummy);
\draw[-latex,thick] (dummy) to [bend right = -30] node[midway, fill=white] {$s^+_{4}$} (4);
\draw[-latex,thick] (4) to [bend right = 0] node[midway, fill=white] {$s^-_{4}$} (dummy);
\draw[-latex, thick,dashed] (3)--(-6,-2) -- node[midway, fill=white] {$s^-_3 = -1$} (-6,6) -- (dummy);
\draw[-latex, thick,dashed] (dummy)--(6,6) -- node[midway, fill=white] {$s^+_3 = -1$} (6,-2) -- (3);
\end{tikzpicture}}
     \end{subfigure}
     \hfill
     \begin{subfigure}[b]{0.475\textwidth}
         \centering
         \subcaptionbox{Undirected Path Circuit (dashed) \label{fig:undir_path_circ}}{
\begin{tikzpicture}[scale = 0.5]
\tikzset{vertex/.style = {shape=circle,draw,minimum size=2em}}
\node[vertex] (dummy) at (0,6) {dummy};
\node[vertex] (1) at (-4,0) {1};
\node[vertex] (2) at (0,2) {2};
\node[vertex] (3) at (0,-2) {3};
\node[vertex] (4) at (4,0) {4};
\draw[-latex,thick,dashed] (1) to node[midway, fill=white] {\tiny$x_{12} = -1$} (2);
\draw[-latex,thick] (1)--(3) {};
\draw[-latex,thick,dashed] (2) to node[midway, fill=white] {\tiny $x_{23} = -1$} (3);
\draw[-latex,thick] (2)--(4) {};
\draw[-latex,thick,dashed] (3) to  node[midway, fill=white] {\tiny $x_{34}=-1$} (4) ;
\draw[-latex,thick] (dummy) to [bend right = 25] node[midway, fill=white] {$s^+_{2}$} (2);
\draw[-latex,thick] (2) to [bend right = 25] node[midway, fill=white] {$s^-_{2}$} (dummy);
\draw[-latex,thick] (dummy) to [bend right = 30] node[yshift = 0.25cm, fill=white] {$s^+_{1}$} (1);
\draw[-latex,thick,dashed] (1) to [bend right = 0] node[midway, fill=white] {$s^-_{1} = 1$} (dummy);
\draw[-latex,thick,dashed] (dummy) to [bend right = -30] node[midway, fill=white] {$s^+_{4} = 1$} (4);
\draw[-latex,thick] (4) to [bend right = 0] node[yshift = -0.25cm, fill=white] {$s^-_{4}$} (dummy);
\draw[-latex, thick] (3)--(-6,-2) -- node[midway, fill=white] {$s^-_3$} (-6,6) -- (dummy);
\draw[-latex, thick] (dummy)--(6,6) -- node[midway, fill=white] {$s^+_3$} (6,-2) -- (3);
\end{tikzpicture}}
     \end{subfigure}
        \caption{An example of a trivial circuit (left) and a path circuit (right), depicted in a pseudoflow network.}
        \label{fig:path_circ}
\end{figure}
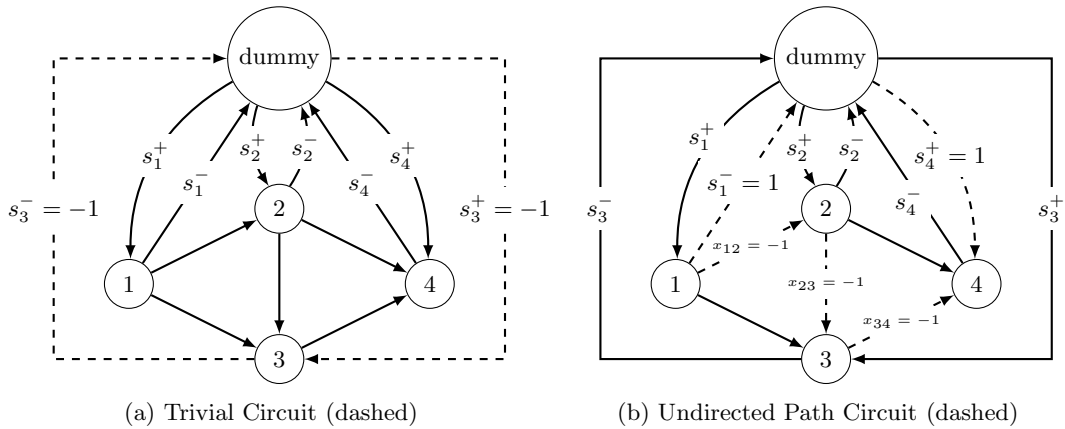

\begin{corollary}\label{cor:final_circ_char}
The circuits of a pseudoflow polyhedron $P$ are:
\begin{itemize}
\item $\pm (\bm{H_{ij}}, \bm{e_i}, \bm{e_j})$, $\pm (\bm{H_{ij}}, \bm{e_i}-\bm{e_j},\bm{0_{n}})$, $\pm(\bm{H_{ij}}, \bm{0_{n}},-\bm{e_i} + \bm{e_j})$, or $\pm(\bm{H_{ij}}, -\bm{e_j}, -\bm{e_i})$ for an undirected path $H_{ij}$ from node $i \in N$ to node $j \in N$ (\textbf{\textit{path circuits}})
    \item $\pm( \bm{C},\bm{0_n},\bm{0_n})$ for an undirected cycle $\bm{C}$ in the original network (\textbf{\textit{cycle circuits}})
    \item $\pm(\bm{0_m}, \bm{e_i}, \bm{e_i})$  
    for all $i \in N$ (\textbf{\textit{trivial circuits}})
\end{itemize}
\end{corollary}

\begin{proof}
By Theorem \ref{thm:pseudo_circs}, we can prove the claim by identifying the vectors corresponding to simple undirected cycles in the pseudoflow network. First, note that cycles $\bm{C}$ in the original network $G$ also are cycles in the pseudoflow network $G_P$. For any cycle, there are two possible orientations. This gives the \textit{cycle circuits} $\pm(\bm{C},\bm{0_n},\bm{0_n})$.

Any other cycle in the pseudoflow network uses two slack arcs and is incident to the dummy node. The \textit{path circuits} of types $\pm(\bm{H_{ij}}, \bm{e_i}, \bm{e_j})$, $\pm(\bm{H_{ij}}, \bm{e_i} -\bm{e_j},\bm{0_{n}})$, $\pm(\bm{H_{ij}}, \bm{0_{n}},-\bm{e_i} + \bm{e_j})$, or $\pm(\bm{H_{ij}}, -\bm{e_j}, -\bm{e_i})$ for an undirected path $H_{ij}$ from $i$ to $j$ correspond to cycles in the pseudoflow network that are completed by connecting $j$ to $i$ through pairs of slack arcs $(s_i^+,s_j^-)$, $(s_i^+,s_j^+)$, $(s_i^-,s_j^-)$, or $(s_j^+,s_i^-)$, respectively.

Finally, cycles of length two through the dummy node and one node $i \in N$ in the original network correspond to vectors $(\bm{0_m}, \bm{e_i}, \bm{e_i})$ and $(\bm{0_m}, -\bm{e_i}, -\bm{e_i})$. We call these \textit{trivial circuits}, as they do not correspond to a path or cycle in the original network.
\end{proof}

The characterization of circuits in Corollary \ref{cor:final_circ_char} will be a crucial tool in our analysis of the algorithms in Section \ref{sec:alg_of_int}. It allows us to relate the iterations of classical network flows algorithms that are described in view of the original network with the steps of an edge walk or general walk over (\ref{pseudo_poly_form}).

\section{Circuit Augmentation Schemes on Pseudoflows}\label{sec:alg_of_int} 
We discuss three combinatorial network flows algorithms with the goal of describing equivalent circuit augmentation schemes over the pseudoflow polyhedron (\ref{pseudo_poly_form}): the Successive Shortest Path Algorithm for min-cost flow problems (Section \ref{sec:sspa}), the Generic or Shortest Augmenting Path Algorithm for max flow problems (Section \ref{sec:GAPA}), and the Hungarian Method for assignment problems (Section \ref{sec:hm}). We identify the types of circuit walks these algorithms perform over (\ref{pseudo_poly_form}) and find a combination of objective function and pivot rule that {\em replicate} a run of the combinatorial algorithm, i.e., that lead to the same circuit walk if used in a generic circuit augmentation scheme. 

We recall some terms and notation that appear in pseudocode for the algorithms; see, e.g., \cite{amo-93}. \textit{Node potentials}, $\bm{\pi}$, correspond to distance labels to represent (a modified) distance of nodes from other specified nodes in the network. An \textit{excess} node is a node $i \in N$ which has a supply value ($b_i>0$), while a \textit{deficit} node is one which has a demand value ($b_i>0$). The term `excess' is also used to describe the current flow imbalance $e(i)$ of node $i$ during the run of an algorithm.  \textit{Reduced costs} for an arc, typically denoted $c_{ij}^{\pi}$, are defined as $c_{ij} - \pi(i) + \pi(j)$; they adjust the arc cost, $c_{ij}$, by the difference of node potentials at the head and tail of the arc. Finally, the \textit{residual capacity}, $r_{ij}$, of an arc refers to the remaining capacity in a network with respect to a current pseudoflow $\bm{x}$, i.e., the capacity of an arc in the corresponding {\em residual network}. For a simple notation, we describe the construction of a residual network based on a pseudoflow as an update of the original network $G(N,A)$.

\subsection{Successive Shortest Path}\label{sec:sspa}

We begin with the Successive Shortest Path Algorithm (SSPA); see Algorithm \ref{alg:SSPA} for a description in pseudocode \cite{o-97}. The algorithm starts by building the initial set $E$ of excess nodes and the set $D$ of deficit nodes (Lines 5-7). Then it enters the main loop. Some excess node, $k \in E$, and some deficit node, $\ell \in D$, are chosen (Line 10). A shortest path $H_{k \ell}$ (with respect to reduced costs) from $k$ to $\ell$ is identified and a flow of $\delta$ is sent along that path (Lines 12-15). The value of $\delta$ depends on the bottleneck of the path, i.e., the capacity of an arc along the path or how much should be sent, $e(k)$, or received, $-e(\ell)$ (Line 14). After flow has been sent, the residual network along with the corresponding sets and reduced costs are updated (Line 16). This process is repeated until no more excess nodes remain. Then there also are no more deficit nodes since the overall excess and deficit in the network must sum to zero for a feasible solution to exist \cite{amo-93}.

\begin{algorithm}
       \caption{Successive Shortest Path Algorithm}\label{alg:SSPA}
\begin{algorithmic}[1] 
       \STATE \mbox{\textbf{Input:} a network $G(N,A)$ with non-negative costs $\bm{c} \in \mathbb{R}^{m}$}
       \STATE \mbox{\textbf{Begin:}}
       \STATE Set arcs $\bm{x} = \bm{0_m}$
       \STATE Set node potentials $\bm{\pi} = \bm{0_n}$
       \STATE Set node potentials $\bm{c^{\pi}} = \bm{c}$
       \STATE Set excess for node $i$, $e(i) = b_i$ for all $i \in N$
       \STATE Initialize $E$ as the set of nodes where $e(i) > 0$
       \STATE Initialize $D$ as the set of nodes where $e(i) < 0$
       \WHILE{$E \neq \emptyset$}
            \STATE Select some $k \in E$ and $\ell \in D$
            \STATE Determine shortest path distances $d(j)$ from node $k$ to all other nodes in $G(N,A)$ with respect to the reduced costs $c_{ij}^{\pi}$
            \STATE Determine the shortest path (with respect to the reduced costs), $H_{k \ell}$, from $k$ to $\ell$
            \STATE Update $\bm{\pi} = \bm{\pi}-\bm{d}$
            \STATE $\delta = \min\{e(k), -e(\ell), \min\{r_{ij}: (i,j) \in H_{k \ell}\}\}$
            \STATE Augment $\delta$ units of flow along path $H_{k \ell}$
            \STATE Update $\bm{x}$, $G(N,A)$, $E$, $D$, and $\bm{c^{\pi}}$
       \ENDWHILE
       \STATE \RETURN{} min-cost flow $\bm{x}$
\end{algorithmic}
\end{algorithm}

Note that reduced costs $c^{\pi}_{ij}$ are used instead of $c_{ij}$ to determine shortest paths (Line 12). Node potentials $\bm{\pi}$ are updated (Line 12) in order to update the reduced costs in each iteration (Line 16). However, the use of node potentials and reduced costs is an optional part of the algorithm and only done for efficiency \cite{o-93}. The difference between the cost of a path $H_{k \ell}$ from $k$ to $\ell$ with respect to reduced costs and the original costs is the term $- \pi(k) + \pi(\ell)$, as
$$\sum_{(i,j) \in H_{k \ell}} c^{\pi}_{ij} = \sum_{(i,j) \in H_{k \ell}} (c_{ij} - \pi(i) + \pi(j)) = \left(\sum_{(i,j) \in H_{k \ell}} c_{ij}\right) - \pi(k) + \pi(\ell).$$
Since the algorithm only compares paths $H_{k \ell}$ between a fixed $k$ and a fixed $\ell$ to one another, $- \pi(k) + \pi(\ell)$ is a constant. Thus, 
$$\arg\min_{H_{k \ell}}\sum_{(i,j) \in H_{k \ell}} c_{ij}=\arg\min_{H_{k \ell}}\sum_{(i,j) \in H_{k \ell}} c^{\pi}_{ij}.$$
This allows us to frame our discussion in terms of the original costs $c_{ij}$.

First, we prove that the SSPA performs a circuit walk over a certain face of (\ref{pseudo_poly_form}). Specifically, we show that a circuit step is taken in each iteration, beginning at the vertex described in Lemma \ref{lem:feas_vert}. 

\begin{lemma}\label{lem:sspa_circ_walk}
The Successive Shortest Path Algorithm performs a circuit walk over the face, $F$, of (\ref{pseudo_poly_form}) where $s_i^+ = 0$ for all $i \in N$ with $b_i<0$, $s_i^- = 0$ for all $i \in N$ with $b_i>0$, and $s_i^\pm =  0$ for all $i \in N$ with $b_i=0$. 
\end{lemma}

\begin{proof}
The SSPA begins with a zero pseudoflow $\bm{x_0}=\bm{0_m}$ (Line 3). We construct the unique vertex of (\ref{pseudo_poly_form}) corresponding to a zero pseudoflow as in Lemma \ref{lem:feas_vert} by setting $s_i^+ = -b_i$ for all $i \in D$, $s_i^- = b_i$ for all $i \in E$, and $0$ otherwise.
Thus, the algorithm begins at a vertex $\bm{y_0}=(\bm{x_0},\bm{s^+},\bm{s^-})$. Note that $\bm{y_0}$ lies in the specified face $F$.

In the first iteration, the algorithm chooses an excess node $k$ and deficit node $\ell$ (Line 10) and finds a shortest path $H_{k \ell}$ between them in the residual network (Line 11). The same flow $\delta$ is sent along each arc on the path (Line 15) and the residual network is updated accordingly (Line 16).

The path $H_{k \ell}$ corresponds to the path circuits in Corollary \ref{thm:pseudo_circs}; adding a (scaled) path circuit to $\bm{y_0}$ is the only way to arrive at a new point $\bm{y_1}=(\bm{x_1},\bm{s_1^+},\bm{s_1^-})$ in (\ref{pseudo_poly_form}) with the desired change in pseudoflow from $\bm{x_0}$ to $\bm{x_1}$ and corresponding slack arcs updating accordingly, $\bm{s_1^+},\bm{s_1^-}$.
While all path circuits are feasible at $\bm{y_0}$ over (\ref{pseudo_poly_form}), only the path circuit $\bm{g}=(\bm{H_{k \ell}},-\bm{e_\ell},-\bm{e_k})$ satisfies $\bm{y_1}=\bm{y_0} + \delta \bm{g}\in F$. 

By choice of $\delta$ (Line 13) as the bottleneck of residual capacities, excess (represented as $s_k^-$), and demand (represented as $s_\ell^+$), a new variable becomes zero or equal to its upper bound in $\bm{y_1}$. This implies that the step to $\bm{y_1}$ is of maximal length. Thus, the first iteration of the algorithm corresponds to a circuit step from $\bm{y_0}$ to $\bm{y_1}$.

Note that slack variables $s_\ell^+$ and $s_k^-$ strictly decreased to $(s'_\ell)^+$ and $(s'_k)^-$, which now correspond to the remaining deficit and excess at nodes $\ell$ and $k$ for $\bm{y_1}$, respectively. Thus, $\bm{y_1}$ again satisfies all properties used in the above arguments for $\bm{y_0}$. Repeating the arguments for the step from $\bm{y_1}$ to $\bm{y_2}$, and later steps, proves the claim.
\end{proof}

Next, we exhibit an example for which the SSPA performs a general (non-edge) walk. Consider the network in Figure \ref{fig:sspa_cycle_ex_orig}. Nodes $1$ to $12$ are labeled with excess and deficit values; the arcs $(i,j)$ are labeled with pairs of capacity and cost $(u_{ij}, c_{ij})$. 

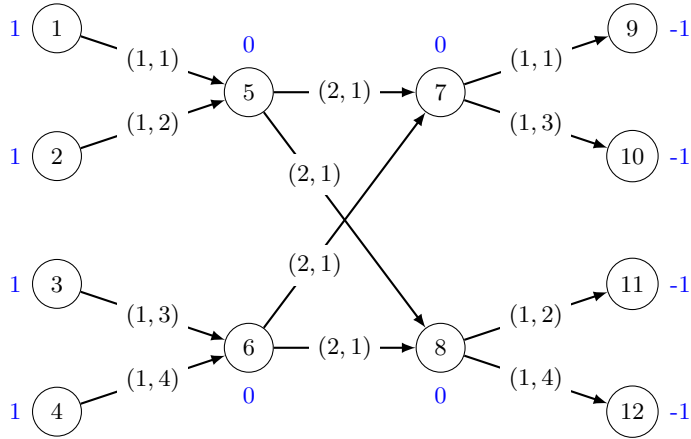
\begin{figure}
    \centering
    \begin{tikzpicture}[scale = 0.85]
        \tikzset{vertex/.style = {shape=circle,draw,minimum size=2em}}
        \tikzset{edge/.style = {-triangle 90,fill=black}}
        \tikzset{edgedd/.style = {dashed,-triangle 90,fill=black}}
        \node[vertex] (1) at (-4,3) {1};
        \node[vertex] (2) at (-4,1) {2};
        \node[vertex] (3) at (-4,-1) {3};
        \node[vertex] (4) at (-4,-3) {4};
        \node[vertex] (5) at (-1,2) {5};
        \node[vertex] (6) at (-1,-2) {6};
        \node[vertex] (7) at (2,2) {7};
        \node[vertex] (8) at (2,-2) {8};
        \node[vertex] (9) at (5,3) {9};
        \node[vertex] (10) at (5,1) {10};
        \node[vertex] (11) at (5,-1) {11};
        \node[vertex] (12) at (5,-3) {12};
        \node[blue] (1s) at (-4.65,3) {1};
        \node[blue] (2s) at (-4.65,1) {1};
        \node[blue] (3s) at (-4.65,-1) {1};
        \node[blue] (4s) at (-4.65,-3) {1};
        \node[blue] (5s) at (-1,2.75) {0};
        \node[blue] (6s) at (-1,-2.75) {0};
        \node[blue] (7s) at (2,2.75) {0};
        \node[blue] (8s) at (2,-2.75) {0};
        \node[blue] (9d) at (5.75,3) {-1};
        \node[blue] (10d) at (5.75,1) {-1};
        \node[blue] (11d) at (5.75,-1) {-1};
        \node[blue] (12d) at (5.75,-3) {-1};
        \draw[-latex,thick] (1)--(5) node[midway, xshift=0mm, yshift = 0mm, fill = white] {$(1,1)$};
        \draw[-latex,thick] (2)--(5) node[midway, xshift=0mm, yshift = 0mm, fill = white] {$(1,2)$};
        \draw[-latex,thick] (3)--(6) node[midway, xshift=0mm, yshift = 0mm, fill = white] {$(1,3)$};
        \draw[-latex,thick] (4)--(6) node[midway, xshift=0mm, yshift = 0mm, fill = white] {$(1,4)$};
        \draw[-latex,thick] (5)--(7) node[midway, xshift=0mm, yshift = 0mm, fill = white] {$(2,1)$};
        \draw[-latex,thick] (6)--(7) node[midway, xshift=-4mm, yshift = -6mm, fill = white] {$(2,1)$};
        \draw[-latex,thick] (5)--(8) node[midway, xshift=-4mm, yshift = 6mm, fill = white] {$(2,1)$};
        \draw[-latex,thick] (6)--(8) node[midway, xshift=0mm, yshift = 0mm, fill = white] {$(2,1)$};
        \draw[-latex,thick] (7)--(9) node[midway, xshift=0mm, yshift = 0mm, fill = white] {$(1,1)$};
        \draw[-latex,thick] (7)--(10) node[midway, xshift=0mm, yshift = 0mm, fill = white] {$(1,3)$};
        \draw[-latex,thick] (8)--(11) node[midway, xshift=0mm, yshift = 0mm, fill = white] {$(1,2)$};
        \draw[-latex,thick] (8)--(12) node[midway, xshift=0mm, yshift = 0mm, fill = white] {$(1,4)$};
    \end{tikzpicture}
    \caption{A network for a run of the Successive Shortest Path Algorithm.}
    \label{fig:sspa_cycle_ex_orig}
\end{figure} 

In each iteration of the SSPA, an excess node $k$ and a demand node $\ell$ are paired. The SSPA works correctly with any rule to choose these pairs $(k,\ell)$. Consider a run in which $(1,9)$, $(2,11)$, $(3,10)$, and $(4,12)$ are chosen in this order. This is not a strong assumption: all excess nodes and all demand nodes have the same values, and any rule based on the node labels would lead to this run by relabeling. In fact, the run also uses the globally shortest paths in each iteration, selected from {\em all} excess and demand node pairs. 

For each pair, a shortest path in the current residual network is established and flow is sent along the path to satisfy the unit excess and demand. The final residual network is shown in Figure \ref{fig:sspa_cycle_ex_res}; arcs are again labeled with capacity and cost $(u_{ij}, c_{ij})$. It contains two 0-cost cycles in opposite directions, shown in dashed and dotted lines. Flow can be sent along either of these cycles without changing the objective function value. This implies that this run of the SSPA does not terminate at a vertex of (\ref{pseudo_poly_form}). In fact, the final flow corresponds to an inner point of an edge of (\ref{MCF}), or the corresponding face of feasible flows in (\ref{pseudo_poly_form}). The fourth step in the run was a general circuit step. Thus, the SSPA may perform a general circuit walk. It also implies that there cannot exist an interpretation of the SSPA as an edge walk over any polyhedron that contains (\ref{MCF}). The example is easy to generalize such that general circuit steps are performed before the final iteration.

\begin{figure}
    \centering
    \begin{tikzpicture}[scale = 0.85]
        \tikzset{vertex/.style = {shape=circle,draw,minimum size=2em}}
        \tikzset{edge/.style = {-triangle 90,fill=black}}
        \tikzset{edgedd/.style = {dashed,-triangle 90,fill=black}}
        \node[vertex] (1) at (-4,3) {1};
        \node[vertex] (2) at (-4,1) {2};
        \node[vertex] (3) at (-4,-1) {3};
        \node[vertex] (4) at (-4,-3) {4};
        \node[vertex] (5) at (-1,2) {5};
        \node[vertex] (6) at (-1,-2) {6};
        \node[vertex] (7) at (2,2) {7};
        \node[vertex] (8) at (2,-2) {8};
        \node[vertex] (9) at (5,3) {9};
        \node[vertex] (10) at (5,1) {10};
        \node[vertex] (11) at (5,-1) {11};
        \node[vertex] (12) at (5,-3) {12};
        \draw[-latex,thick] (5)--(1) node[midway, xshift=0mm, yshift = 0mm, fill = white] {$(1,-1)$};
        \draw[-latex,thick] (5)--(2) node[midway, xshift=0mm, yshift = 0mm, fill = white] {$(1,-2)$};
        \draw[-latex,thick] (6)--(3) node[midway, xshift=0mm, yshift = 0mm, fill = white] {$(1,-3)$};
        \draw[-latex,thick] (6)--(4) node[midway, xshift=0mm, yshift = 0mm, fill = white] {$(1,-4)$};
        
        \draw[-latex,thick,dashed] (5)--(7) node[midway, xshift=0mm, yshift = 0mm, fill = white] {$(1,1)$};
        \draw[-latex,thick,dotted] (5)--(8) node[midway, xshift=-3.5mm, yshift = 6mm, fill = white] {$(1,1)$};
        \draw[-latex,thick,dashed] (6)--(8) node[midway, xshift=0mm, yshift = 0mm, fill = white] {$(1,1)$};
        \draw[-latex,thick,dotted] (6)--(7) node[midway, xshift=-4mm, yshift = -6mm, fill = white] {$(1,1)$};
        \draw[-latex,thick,dotted] (7) to [bend right = 35] node[midway, xshift=0mm, yshift = 1mm, fill = white] {$(1,-1)$} (5);
        \draw[-latex,thick,dashed] (7) to [bend right = 45] node[midway, xshift=-4mm, yshift = -7mm, fill = white] {$(1,-1)$} (6);
        \draw[-latex,thick,dashed] (8) to [bend right = 45] node[midway, xshift=4mm, yshift = -7mm, fill = white] {$(1,-1)$} (5);
        \draw[-latex,thick,dotted] (8) to [bend right = -35] node[midway, xshift=0mm, yshift = -1mm, fill = white] {$(1,-1)$} (6);
        \draw[-latex,thick] (9)--(7) node[midway, xshift=0mm, yshift = 0mm, fill = white] {$(1,-1)$};
        \draw[-latex,thick] (10)--(7) node[midway, xshift=0mm, yshift = 0mm, fill = white] {$(1,-3)$};
        \draw[-latex,thick] (11)--(8) node[midway, xshift=0mm, yshift = 0mm, fill = white] {$(1,-2)$};
        \draw[-latex,thick] (12)--(8) node[midway, xshift=0mm, yshift = 0mm, fill = white] {$(1,-4)$};
    \end{tikzpicture}
    \caption{A residual network in a run of the Successive Shortest Path Algorithm. There are two opposite 0-cost cycles shown in dashed and dotted lines.}
    \label{fig:sspa_cycle_ex_res}
\end{figure}
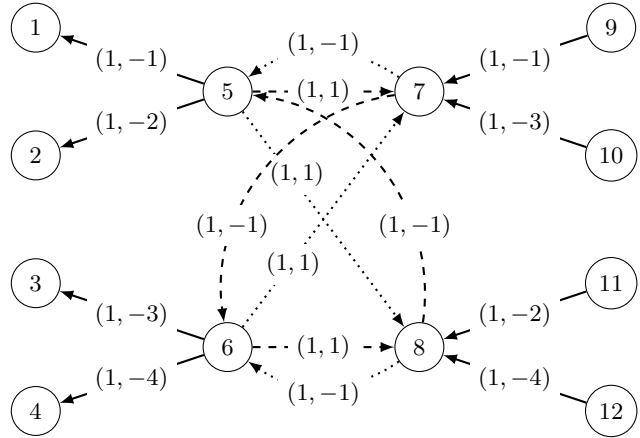

We conclude the discussion of the SSPA by showing that, for a certain objective function, a circuit augmentation scheme using Dantzig's rule replicates a run of the algorithm. It leads to a circuit walk as in Lemma \ref{lem:sspa_circ_walk}. Note that this walk lies in the face $F$ of (\ref{pseudo_poly_form}) with $s_i^+=0$ for all nodes $i\notin D$ and $s_i^-=0$ for all nodes $i\notin E$. By restricting to this face, the correspondence of a path $H_{k \ell}$ and a feasible path circuit becomes unique. For a circuit augmentation scheme, the restriction to face $F$ can be realized by explicitly describing $F$ instead of (\ref{pseudo_poly_form}), i.e., by adding the constraints forcing variables to $0$ or not introducing those variables, or by imposing a penalty term on slack variables that should remain $0$ in a scheme over the whole polyhedron (\ref{pseudo_poly_form}). We use the latter approach to describe the augmentation scheme. 

\begin{theorem}\label{thm:sspa_dantzig}
A circuit augmentation scheme over (\ref{pseudo_poly_form}) using Dantzig's rule replicates a run of the Successive Shortest Path Algorithm. 
\end{theorem}

\begin{proof}
We prove the claim by providing an objective function $\bm{c}=(\bm{c_x},\bm{c_{s^+}},\bm{c_{s^-}})\in \mathbb{R}^{m+2n}$ such that the choice of a Dantzig circuit $\bm{g}$ corresponds to the selection of a globally shortest path across all pairs of excess and demand nodes. Here, $\bm{c_x}\in \mathbb{R}^m$ denotes the original (non-negative) arc costs in the input of Algorithm \ref{alg:SSPA}. As we will see, it suffices to provide appropriate costs $\bm{c_{s^+}},\bm{c_{s^-}}$ for the slack variables. 

The SSPA iteratively reduces the excess and demand values in the network and terminates with a feasible flow for (\ref{MCF}), i.e., a pseudoflow of the form $(\bm{x},\bm{0_n},\bm{0_n})$. To mirror this behavior, we penalize flow on slack arcs by introducing a sufficiently large cost $M$, for example $M = 1 + \sum_{(i,j) \in A} c_{ij}$, and setting $\bm{c}=(\bm{c_x},\bm{c_{s^+}},\bm{c_{s^-}})=(\bm{c_x},\bm{M_{n}},\bm{M_{n}})$, where $\bm{M_{n}}= M\cdot \bm{1_{n}}$. 

Recall that Dantzig's rule is the choice of a feasible circuit $\bm{g}$ such that $-\bm{c^T}\bm{g}$ is maximized. We consider the values $\bm{c^T}\bm{g} = \bm{c_x^Tx} + \bm{M_{n}^Ts^+}+\bm{M_{n}^Ts^-}$ for the three types of circuits $\bm{g}=(\bm{x},\bm{s^+},\bm{s^-})$ from Corollary \ref{cor:final_circ_char}. Note that $-M < \bm{c_x^Tg_x} < M$ for any circuit $g$, where $\bm{g}=(\bm{g_x},\bm{g_{s^+}},\bm{g_{s^-}})$ and $\bm{g_x}$ are the indices of the circuit $g$ that corresponds to $x$-arcs. This allows a categorization of the circuits by the `slack contribution' $s_g=\bm{M_{n}^Ts^+}+\bm{M_{n}^Ts^-}$: trivial circuits have $s_g=2M$ or $s_g=-2M$; cycle circuits have $s_g = 0$; path circuits have $s_g=-2M$, $s_g=0$, or $s_g=2M$. If there exists a circuit with slack contribution $-2M$ -- which corresponds to reducing flow along two slack arcs -- a maximum value of $-\bm{c^Tg}$ can only be achieved for such a circuit. This can only be a path circuit or a trivial circuit. 

The circuit augmentation scheme starts at the vertex $\bm{y_0}=(\bm{x_0},\bm{s^+},\bm{s^-})$ with $\bm{x_0}=\bm{0_m}$ and at least one of each pair $s_i^\pm$ equal to zero; see Lemma \ref{lem:feas_vert}. Because of the latter, trivial circuits are not feasible at $\bm{y_0}$. For any pair of excess node $k$ and deficit node $\ell$, there exists a unique path circuit $\bm{g}=(\bm{H_{k \ell}},-\bm{e_\ell},-\bm{e_k})$ that reduces flow along two slack arcs. Dantzig's rule chooses a best path circuit $\bm{g}$ among these and performs a maximal circuit step to $\bm{y_1}=\bm{y_0} + \delta \bm{g}$. 

As in the proof of Lemma \ref{lem:sspa_circ_walk}, $\bm{y_1}$ has at least one of each pair $s_i^\pm$ equal to zero and a non-zero slack variable corresponds to the remaining excess or deficit at the node. This implies that the above arguments can be repeated for $\bm{y_{i+1}}$ in place of $\bm{y_i}$. Thus trivial circuits are not feasible while there remains an excess (and a deficit) node since a path circuit that reduces flow along two non-zero slack arcs will be chosen. This implies the new $\bm{y_{i+2}}$ satisfies the same properties again.
\end{proof}

\subsection{Generic Augmenting Path}\label{sec:GAPA}

The Generic Augmenting Path Algorithm (GAPA) is a classical algorithm for solving max flow problems. One aims to send as much flow as possible from a source node $s$ to a sink node $t$; see Algorithm \ref{GAP} for a description in pseudocode \cite{o-97}. The algorithm begins by finding an augmenting $s-t$ path $H$, i.e., a path from $s$ to $t$ (Line 5) in the current residual network, and sends a maximal flow of $\delta$ along the path $H$ (Lines 6-7). Then the algorithm updates the residual network (Line 8) and the process is repeated until no more $s-t$ paths exist. 

It is well-known that max flow problems are a special case of min-cost flow problems and, as we will see, its iterations are readily interpreted as augmentations on a current pseudoflow. Max flow problems can be formulated as circulation problems, i.e., min-cost flow problems where $b_i=0$ for all $i \in N$, with the addition of the variable $x_{ts}$, which corresponds to flow along an imaginary arc from $t$ to $s$ in the original network. The flow on this arc represents the total flow that reaches $t$ from $s$ in the original network.

A pseudoflow network can be seen as a more general construction that captures this reformulation as a special case. We mimic the behavior in the pseudoflow network through the pair of slack arcs for $s^-_{t}$ and $s^+_{s}$ taking the place of $x_{ts}$.

Due to this formulation, all intermediate and the final solution of a run of the GAPA are of the form $(\bm{x},\alpha \bm{e_s}, \alpha \bm{e_t})$, where $\alpha$ is the sum total of flow which has currently reached $t$. Further, the starting vertex for the GAPA is a special case of the vertex described in Lemma \ref{lem:feas_vert}, a zero flow $\bm{0_{m+2n}}$. In a circulation problem, all supply and demand values for nodes are zero, so there is no initial flow that needs to be placed on the slack arcs. While the starting vertex and form of solutions differ from the SSPA, several arguments in Section \ref{sec:sspa} carry over and we are able to follow the same motions, with a number of differences in technical details. 

\begin{algorithm}
       \caption{Generic Augmenting Path Algorithm}\label{GAP}
\begin{algorithmic}[1] 
       \STATE \mbox{\textbf{Input:} a network $G(N,A)$ with non-negative capacities $u_{ij}$, source node $s$, sink node $t$}
       \STATE \mbox{\textbf{Begin:}}
       \STATE Set arcs $\bm{x} = \bm{0_m}$
       \WHILE{there exists a path from $s$ to $t$}
            \STATE Determine an augmenting path, $H_{st}$, from $s$ to $t$
            \STATE $\delta =\min\{r_{ij}: (i,j) \in H_{st}\}$
            \STATE Augment $\delta$ units of flow along path $H_{st}$
            \STATE Update $\bm{x}$, $G(N,A)$
       \ENDWHILE
       \STATE \RETURN{} max flow $\bm{x}$
\end{algorithmic}
\end{algorithm}

We begin by proving that the GAPA performs a circuit walk over a face of (\ref{pseudo_poly_form}) that restricts to pseudoflows of the form $(\bm{x},\alpha \bm{e_s}, \alpha \bm{e_t})$. In particular, we show that a circuit step is taken in each iteration beginning at the vertex $\bm{0_{m+2n}}$. 

\begin{lemma}\label{lem:gapa_circ_walk}
    The Generic Augmenting Path Algorithm performs a circuit walk over the face, $F$, of (\ref{pseudo_poly_form}) where $s_i^\pm =  0$ for all $i \in N \setminus\{s,t\}$ and $s^-_s = s^+_t=0$. (Only $s^+_s$ and $s^-_t$ are allowed to be non-zero.) 
\end{lemma}

\begin{proof}
The GAPA begins by setting all $x$-arcs equal to zero, $\bm{x_0} = \bm{0_m}$ (Line 3). With $b_i=0$ for all $i\in N$, by Lemma \ref{lem:feas_vert}, the unique vertex of (\ref{pseudo_poly_form}) corresponding to a zero pseudoflow is $\bm{y_0} = (\bm{x_0}, \bm{s^+}, \bm{s^-}) = \bm{0_{m+2n}}$. Note that $\bm{y_0}$ lies in the specified face $F$.

In each iteration, augmentation along an $s-t$ path $H_{st}$ in the current residual network is performed (Lines 5-7). The iterations of the GAPA mirror the iterations of the SSPA where $s$ and $t$ would always be chosen as the pair of supply and demand nodes $k$ and $\ell$ in the SSPA, which allows us to argue similarly to the proof of Lemma \ref{lem:sspa_circ_walk}.
 
At $\bm{y_0}$, only the path circuit $\bm{g}=(\bm{H^1_{st}},\bm{e_s},\bm{e_t})$ is feasible (all others would have to decrease a slack variable below $0$) and one has $\bm{y_1}=\bm{y_0} + \delta \bm{g}=(\bm{x}_1, \delta \bm{e_s},\delta \bm{e_s})\in F$ where $\bm{x_1}$ corresponds to the updated flow values for arcs along the path $H^1_{st}$. By choice of $\delta$ as the bottleneck of residual capacities (Line 6), a new variable becomes zero or equal to its upper bound in $\bm{x_1}$. Thus, the step from $\bm{y_0}$ to $\bm{y_1}$ is of maximal length.

Due to $\bm{g}$ only increasing flow along slack arcs $s^+_s$ and $s^-_t$, all slack arcs aside from that pair remain zero. Therefore, at $\bm{y_1}$ again only a path circuit of the form $(\bm{H^2_{st}},\bm{e_s},\bm{e_t})$ is feasible for the next augmenting $s-t$ path $H^2_{st}$; a step of maximal length is performed. The same arguments can be repeated for subsequent steps, which proves the claim. 
\end{proof}

Let us provide an example for which the GAPA performs a general (non-edge) circuit walk.  
Consider the network in Figure \ref{fig:gapa_cycle_ex_orig}; the arcs are labeled with capacity values $u_{ij}$. The network is constructed from the one in Figure \ref{fig:sspa_cycle_ex_orig}: new nodes $s$ and $t$ are added and connected to excess and demand nodes through new unit-capacity arcs, respectively. Further, any arc incident to excess or demand nodes is replaced by a path of unit-capacity arcs with a length equal to the respective cost.

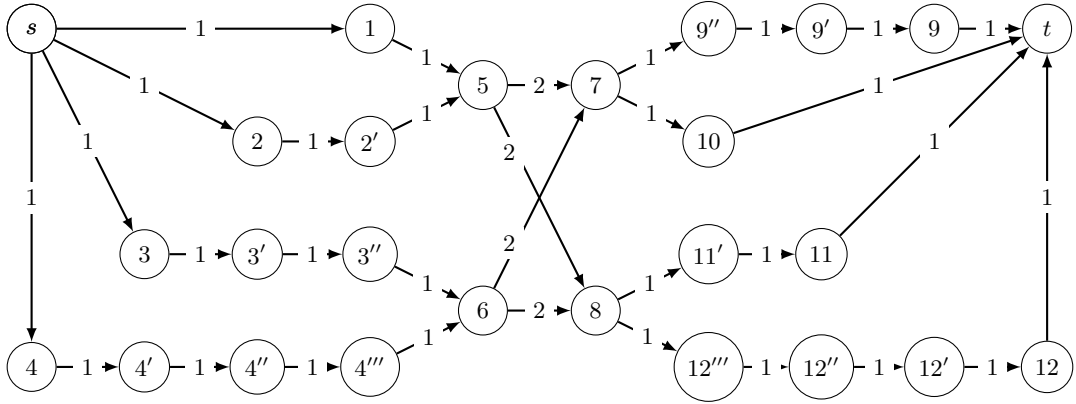
\begin{figure}
    \centering
    \begin{tikzpicture}[scale = 0.75]
        \tikzset{vertex/.style = {shape=circle,draw,minimum size=2em}}
        \tikzset{edge/.style = {-triangle 90,fill=black}}
        \tikzset{edgedd/.style = {dashed,-triangle 90,fill=black}}
        \node[vertex] (s) at (-9,3) {$s$};
        \node[vertex] (s) at (-9,3) {$s$};
        \node[vertex] (t) at (9,3) {$t$};
        \node[vertex] (1) at (-3,3) {1};
        \node[vertex] (2) at (-5,1) {2};
        \node[vertex] (2?) at (-3,1) {$2'$};
        \node[vertex] (3) at (-7,-1) {3};
        \node[vertex] (3?) at (-5,-1) {$3'$};
        \node[vertex] (3??) at (-3,-1) {$3''$};
        \node[vertex] (4) at (-9,-3) {4};
        \node[vertex] (4?) at (-7,-3) {$4'$};
        \node[vertex] (4??) at (-5,-3) {$4''$};
        \node[vertex] (4???) at (-3,-3) {$4'''$};
        \node[vertex] (5) at (-1,2) {5};
        \node[vertex] (6) at (-1,-2) {6};
        \node[vertex] (7) at (1,2) {7};
        \node[vertex] (8) at (1,-2) {8};
        \node[vertex] (9) at (7,3) {9};
        \node[vertex] (9?) at (5,3) {$9'$};
        \node[vertex] (9??) at (3,3) {$9''$};
        \node[vertex] (10) at (3,1) {10};
        \node[vertex] (11) at (5,-1) {11};
        \node[vertex] (11?) at (3,-1) {$11'$};
        \node[vertex] (12) at (9,-3) {12};
        \node[vertex] (12?) at (7,-3) {$12'$};
        \node[vertex] (12??) at (5,-3) {$12''$};
        \node[vertex] (12???) at (3,-3) {$12'''$};
        \draw[-latex,thick] (s)--(1) node[midway, xshift=0mm, yshift = 0mm, fill = white] {$1$};
        \draw[-latex,thick] (s)--(2) node[midway, xshift=0mm, yshift = 0mm, fill = white] {$1$};
        \draw[-latex,thick] (s)--(3) node[midway, xshift=0mm, yshift = 0mm, fill = white] {$1$};
        \draw[-latex,thick] (s)--(4) node[midway, xshift=0mm, yshift = 0mm, fill = white] {$1$};
        
        \draw[-latex,thick] (1)--(5) node[midway, xshift=0mm, yshift = 0mm, fill = white] {$1$};
        
        \draw[-latex,thick] (2)--(2?) node[midway, xshift=0mm, yshift = 0mm, fill = white] {$1$};
        \draw[-latex,thick] (2?)--(5) node[midway, xshift=0mm, yshift = 0mm, fill = white] {$1$};
        
        \draw[-latex,thick] (3)--(3?) node[midway, xshift=0mm, yshift = 0mm, fill = white] {$1$};
        \draw[-latex,thick] (3?)--(3??) node[midway, xshift=0mm, yshift = 0mm, fill = white] {$1$};
        \draw[-latex,thick] (3??)--(6) node[midway, xshift=0mm, yshift = 0mm, fill = white] {$1$};
        
        \draw[-latex,thick] (4)--(4?) node[midway, xshift=0mm, yshift = 0mm, fill = white] {$1$};
        \draw[-latex,thick] (4?)--(4??) node[midway, xshift=0mm, yshift = 0mm, fill = white] {$1$};
        \draw[-latex,thick] (4??)--(4???) node[midway, xshift=0mm, yshift = 0mm, fill = white] {$1$};
        \draw[-latex,thick] (4???)--(6) node[midway, xshift=0mm, yshift = 0mm, fill = white] {$1$};
        
        \draw[-latex,thick] (5)--(7) node[midway, xshift=0mm, yshift = 0mm, fill = white] {$2$};
        \draw[-latex,thick] (6)--(7) node[midway, xshift=-4mm, yshift = -6mm, fill = white] {$2$};
        \draw[-latex,thick] (5)--(8) node[midway, xshift=-4mm, yshift = 6mm, fill = white] {$2$};
        \draw[-latex,thick] (6)--(8) node[midway, xshift=0mm, yshift = 0mm, fill = white] {$2$};
        
        \draw[-latex,thick] (7)--(9??) node[midway, xshift=0mm, yshift = 0mm, fill = white] {$1$};
        \draw[-latex,thick] (9??)--(9?) node[midway, xshift=0mm, yshift = 0mm, fill = white] {$1$};
        \draw[-latex,thick] (9?)--(9) node[midway, xshift=0mm, yshift = 0mm, fill = white] {$1$};
        
        \draw[-latex,thick] (7)--(10) node[midway, xshift=0mm, yshift = 0mm, fill = white] {$1$};
        
        \draw[-latex,thick] (8)--(11?) node[midway, xshift=0mm, yshift = 0mm, fill = white] {$1$};
        \draw[-latex,thick] (11?)--(11) node[midway, xshift=0mm, yshift = 0mm, fill = white] {$1$};
        
        \draw[-latex,thick] (8)--(12???) node[midway, xshift=0mm, yshift = 0mm, fill = white] {$1$};
        \draw[-latex,thick] (12???)--(12??) node[midway, xshift=0mm, yshift = 0mm, fill = white] {$1$};
        \draw[-latex,thick] (12??)--(12?) node[midway, xshift=0mm, yshift = 0mm, fill = white] {$1$};
        \draw[-latex,thick] (12?)--(12) node[midway, xshift=0mm, yshift = 0mm, fill = white] {$1$};

        \draw[-latex,thick] (9)--(t) node[midway, xshift=0mm, yshift = 0mm, fill = white] {$1$};
        \draw[-latex,thick] (10)--(t) node[midway, xshift=0mm, yshift = 0mm, fill = white] {$1$};
        \draw[-latex,thick] (11)--(t) node[midway, xshift=0mm, yshift = 0mm, fill = white] {$1$};
        \draw[-latex,thick] (12)--(t) node[midway, xshift=0mm, yshift = 0mm, fill = white] {$1$};
    \end{tikzpicture}
    \caption{A network for a run of the Generic Augmenting Path Algorithm.}
    \label{fig:gapa_cycle_ex_orig}
\end{figure}

\begin{figure}
    \centering
    \begin{tikzpicture}[scale = 0.75]
        \tikzset{vertex/.style = {shape=circle,draw,minimum size=2em}}
        \tikzset{edge/.style = {-triangle 90,fill=black}}
        \tikzset{edgedd/.style = {dashed,-triangle 90,fill=black}}
        \node[vertex] (s) at (-9,3) {$s$};
        \node[vertex] (s) at (-9,3) {$s$};
        \node[vertex] (t) at (9,3) {$t$};
        \node[vertex] (1) at (-3,3) {1};
        \node[vertex] (2) at (-5,1) {2};
        \node[vertex] (2?) at (-3,1) {$2'$};
        \node[vertex] (3) at (-7,-1) {3};
        \node[vertex] (3?) at (-5,-1) {$3'$};
        \node[vertex] (3??) at (-3,-1) {$3''$};
        \node[vertex] (4) at (-9,-3) {4};
        \node[vertex] (4?) at (-7,-3) {$4'$};
        \node[vertex] (4??) at (-5,-3) {$4''$};
        \node[vertex] (4???) at (-3,-3) {$4'''$};
        \node[vertex] (5) at (-1,2) {5};
        \node[vertex] (6) at (-1,-2) {6};
        \node[vertex] (7) at (1,2) {7};
        \node[vertex] (8) at (1,-2) {8};
        \node[vertex] (9) at (7,3) {9};
        \node[vertex] (9?) at (5,3) {$9'$};
        \node[vertex] (9??) at (3,3) {$9''$};
        \node[vertex] (10) at (3,1) {10};
        \node[vertex] (11) at (5,-1) {11};
        \node[vertex] (11?) at (3,-1) {$11'$};
        \node[vertex] (12) at (9,-3) {12};
        \node[vertex] (12?) at (7,-3) {$12'$};
        \node[vertex] (12??) at (5,-3) {$12''$};
        \node[vertex] (12???) at (3,-3) {$12'''$};
        
        \draw[-latex,thick] (1)--(s) node[midway, xshift=0mm, yshift = 0mm, fill = white] {$1$};
        \draw[-latex,thick] (2)--(s) node[midway, xshift=0mm, yshift = 0mm, fill = white] {$1$};
        \draw[-latex,thick] (3)--(s) node[midway, xshift=0mm, yshift = 0mm, fill = white] {$1$};
        \draw[-latex,thick] (4)--(s) node[midway, xshift=0mm, yshift = 0mm, fill = white] {$1$};
        
        \draw[-latex,thick] (5)--(1) node[midway, xshift=0mm, yshift = 0mm, fill = white] {$1$};
        
        \draw[-latex,thick] (2?)--(2) node[midway, xshift=0mm, yshift = 0mm, fill = white] {$1$};
        \draw[-latex,thick] (5)--(2?) node[midway, xshift=0mm, yshift = 0mm, fill = white] {$1$};
        
        \draw[-latex,thick] (3?)--(3) node[midway, xshift=0mm, yshift = 0mm, fill = white] {$1$};
        \draw[-latex,thick] (3??)--(3?) node[midway, xshift=0mm, yshift = 0mm, fill = white] {$1$};
        \draw[-latex,thick] (6)--(3??) node[midway, xshift=0mm, yshift = 0mm, fill = white] {$1$};
        
        \draw[-latex,thick] (4?)--(4) node[midway, xshift=0mm, yshift = 0mm, fill = white] {$1$};
        \draw[-latex,thick] (4??)--(4?) node[midway, xshift=0mm, yshift = 0mm, fill = white] {$1$};
        \draw[-latex,thick] (4???)--(4??) node[midway, xshift=0mm, yshift = 0mm, fill = white] {$1$};
        \draw[-latex,thick] (6)--(4???) node[midway, xshift=0mm, yshift = 0mm, fill = white] {$1$};
        
        \draw[-latex,thick,dashed] (5)--(7) node[midway, xshift=0mm, yshift = 0mm, fill = white] {$1$};
        \draw[-latex,thick,dotted] (6)--(7) node[midway, xshift=-4mm, yshift = -6mm, fill = white] {$1$};
        \draw[-latex,thick,dotted] (5)--(8) node[midway, xshift=-4mm, yshift = 6mm, fill = white] {$1$};
        \draw[-latex,thick,dashed] (6)--(8) node[midway, xshift=0mm, yshift = 0mm, fill = white] {$1$};
        
        \draw[-latex,thick] (9??)--(7) node[midway, xshift=0mm, yshift = 0mm, fill = white] {$1$};
        \draw[-latex,thick] (9?)--(9??) node[midway, xshift=0mm, yshift = 0mm, fill = white] {$1$};
        \draw[-latex,thick] (9)--(9?) node[midway, xshift=0mm, yshift = 0mm, fill = white] {$1$};
        
        \draw[-latex,thick] (10)--(7) node[midway, xshift=0mm, yshift = 0mm, fill = white] {$1$};
        
        \draw[-latex,thick] (11?)--(8) node[midway, xshift=0mm, yshift = 0mm, fill = white] {$1$};
        \draw[-latex,thick] (11)--(11?) node[midway, xshift=0mm, yshift = 0mm, fill = white] {$1$};
        
        \draw[-latex,thick] (12???)--(8) node[midway, xshift=0mm, yshift = 0mm, fill = white] {$1$};
        \draw[-latex,thick] (12??)--(12???) node[midway, xshift=0mm, yshift = 0mm, fill = white] {$1$};
        \draw[-latex,thick] (12?)--(12??) node[midway, xshift=0mm, yshift = 0mm, fill = white] {$1$};
        \draw[-latex,thick] (12)--(12?) node[midway, xshift=0mm, yshift = 0mm, fill = white] {$1$};

        \draw[-latex,thick] (t)--(9) node[midway, xshift=0mm, yshift = 0mm, fill = white] {$1$};
        \draw[-latex,thick] (t)--(10) node[midway, xshift=0mm, yshift = 0mm, fill = white] {$1$};
        \draw[-latex,thick] (t)--(11) node[midway, xshift=0mm, yshift = 0mm, fill = white] {$1$};
        \draw[-latex,thick] (t)--(12) node[midway, xshift=0mm, yshift = 0mm, fill = white] {$1$};

        \draw[-latex,thick,dotted] (7) to [bend right = 35] node[midway, xshift=0mm, yshift = 1mm, fill = white] {$1$} (5);
        \draw[-latex,thick,dotted] (8) to [bend right = -35] node[midway, xshift=0mm, yshift = -1mm, fill = white] {$1$} (6);
        \draw[-latex,thick,dashed] (7) to [bend right = 45] node[midway, xshift=0mm, yshift = 1mm, fill = white] {$1$} (6);
        \draw[-latex,thick,dashed] (8) to [bend right = 45] node[midway, xshift=0mm, yshift = 1mm, fill = white] {$1$} (5);
    \end{tikzpicture}
    \caption{A residual network in a run of the Generic Augmenting Path Algorithm. There are two opposite 0-costs cycles shown in dashed and dotted lines.}
    \label{fig:gapa_cycle_ex_res}
\end{figure}
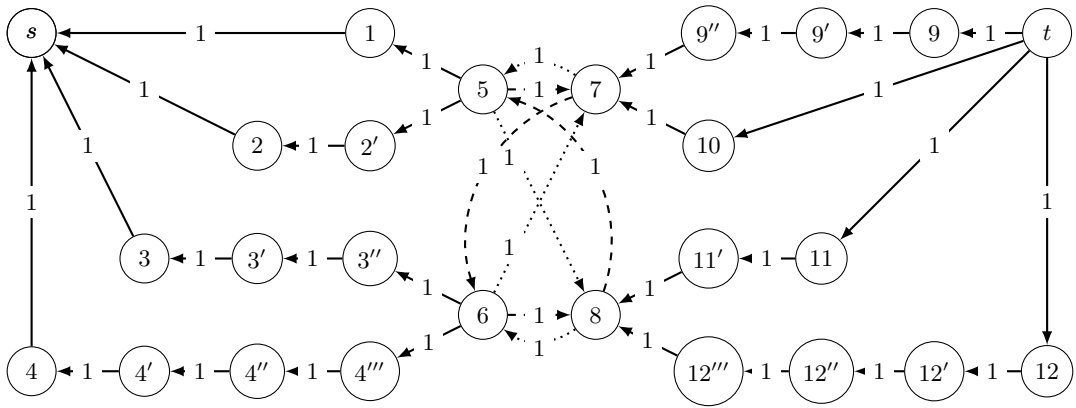

The design of the network leads to iterations that are in one-to-one correspondence with those in the previous SSPA example. In each iteration, an $s-t$ path is found and used for augmentation. If shortest paths are used, a run would first choose the $s-t$ path through the pair of nodes $1$ and $10$, denoted (1,10). The subsequent pairings would be (2,11), (3,9), and (4,12). The capacity of the augmenting paths is always one and the run terminates after four augmentations.

Such a run creates two 0-cost cycles in opposite directions, shown in dashed and dotted lines in Figure \ref{fig:gapa_cycle_ex_res}. In fact, these cycles correspond to the ones formed in the SSPA example, and the same arguments suffice to see that the GAPA may terminate at a non-vertex of (\ref{pseudo_poly_form}). Thus, the GAPA may perform a general circuit walk and there cannot exist an interpretation of the GAPA as an edge walk over any polyhedron that contains (\ref{MCF}).

Next, we describe combinations of objective functions and pivot rules to obtain circuit augmentation schemes over  (\ref{pseudo_poly_form}) that replicate a run of the GAPA and the Shortest Augmenting Path Algorithm (SAPA), the special variant named for the use of shortest $s-t$ paths. As before, we use a penalty term on slack variables that should remain $0$ in a scheme defined over the whole polyhedron. The resulting walk lies in the face specified in Lemma \ref{lem:gapa_circ_walk}.

We begin by exhibiting an objective function for which any pivot rule replicates a run of the GAPA. It builds on the standard way to represent a max flow objective in a reformulation as a circulation problem, adjusted to the specific form of (\ref{pseudo_poly_form}). 

\begin{theorem}\label{thm:gap_aug}
A circuit augmentation scheme over (\ref{pseudo_poly_form}) using any pivot rule replicates a run of the Generic Augmenting Path Algorithm.
\end{theorem}

\begin{proof}
We prove the claim by providing an objective function $\bm{c}=(\bm{0_m}, \bm{c_{s^+}}, \bm{c_{s^-}}) \in \R^{m+2n}$ such that any pivot rule would select an augmenting $s-t$ path in each iteration.  
To reproduce this behavior, we reward circuits which use the arcs $s^+_s$ and $s^-_t$ and penalize the use of other slack arcs. Further, we set the costs of the $x$-arcs in the original network to zero, which implies that all cycle circuits have cost $0$.

Placing a cost of $1$ on arcs $s^+_s$ and $s^-_t$ is sufficient such that any $s-t$ path circuit $H_{st}$ using those slack arcs is of improving cost $2$. The other slack arcs are penalized by a sufficiently large negative cost $-M$; for example $M>3$. Using $\bm{M_{n}}= M\cdot \bm{1_{n}}$, one can write the cost function as $\bm{c} = (\bm{0_m}, \bm{-M_n}+(1+M)\bm{e}_s, \bm{-M_n}+(1+M)\bm{e}_t)$.

By design of $\bm{c}$, trivial circuits or path circuits (other than $H_{st}$) are of improving cost only if flow is reduced along penalized slack arcs. However, such circuits do not become feasible throughout the run; recall the proof of Lemma \ref{lem:gapa_circ_walk}. As the circuit augmentation scheme starts at the vertex $\bm{y_0}=(\bm{x_0}, \bm{s^+}, \bm{s^-}) = \bm{0_{m+2n}}$, in the first iteration the only improving, feasible circuits are $s-t$ path circuits with slack arcs $s^+_s$ and $s^-_t$. All other slack arcs remain zero. Thus, the same argument holds for later iterations, too. The scheme continues to find improving $s-t$ path circuits of cost $2$ as long as there remain augmenting $s-t$ paths in the original network. This proves the claim. 
\end{proof}

For the objective function used in this proof, all improving feasible circuits are $s-t$ path circuits of cost $2$. This allows the selection of a shortest path through a steepest-ascent rule: the cost $2$ of the circuit $\bm{g}$ is divided by the number $||B\bm{g}||_1$ of arcs of $\bm{g}$, which leads to a maximal improvement $\bm{c}^T\bm{g}/||B\bm{g}||_1$ for a shortest path. We obtain the following corollary for the SAPA.

\begin{corollary}\label{cor:sapa_saa}
A circuit augmentation scheme over (\ref{pseudo_poly_form}) using the steepest-ascent pivot rule replicates a run of the Shortest Augmenting Path Algorithm.      
\end{corollary}
 
Steepest-ascent circuit augmentation is known to have an efficient implementation \cite{bv-19b}, which allows us to provide a proof-of-concept implementation of the SAPA in this form. The code can be found at \url{https://github.com/angela-r-morrison/pseudoflow_polyhedron}. In Figure \ref{fig:sapa_saa_its}, we display the iterations for a small example in the repository. Figure \ref{fig:orig_net_SAPA} shows the original network with arc capacities. The steepest-ascent circuits or shortest paths are displayed as dashed edges.
 
    \begin{figure}[t]
        \centering
        \begin{subfigure}[b]{0.5\textwidth}
         \centering
\subcaptionbox{Original Network with Capacities\label{fig:orig_net_SAPA}}{
\begin{tikzpicture}[scale = 0.85]
\tikzset{vertex/.style = {shape=circle,draw,minimum size=1em}}
\node[vertex]  (1) at (0,0) {1};
\node[vertex]  (2) at (2,1.25) {2};
\node[vertex]  (3) at (4,1.25) {3};
\node[vertex]  (4) at (2,-1.25) {4};
\node[vertex]  (5) at (4,-1.25) {5};
\node[vertex]  (6) at (6,0) {6};
\draw[-latex,ultra thick](1)--(2) node[midway, xshift=-2mm, yshift = 0mm, fill = white] {6};
\draw[-latex,ultra thick](1)--(4) node[midway, xshift=-2mm, yshift = 0mm, fill = white] {6};
\draw[-latex,ultra thick](2)--(3) node[midway, xshift=0mm, yshift = 0mm, fill = white] {3};
\draw[-latex,ultra thick](2)--(4) node[midway, xshift=0mm, yshift = 0mm, fill = white] {7};
\draw[-latex,ultra thick](3)--(5) node[midway, xshift=0mm, yshift = 0mm, fill = white] {2};
\draw[-latex,ultra thick](3)--(6) node[midway, xshift=-2mm, yshift = 0mm, fill = white] {7};
\draw[-latex,ultra thick](4)--(3) node[midway, xshift=0mm, yshift = 0mm, fill = white] {10};
\draw[-latex,ultra thick](4)--(5) node[midway, xshift=0mm, yshift = 0mm, fill = white] {2};
\draw[-latex,ultra thick](5)--(6) node[midway, xshift=-2mm, yshift = 0mm, fill = white] {4};
\end{tikzpicture}}
     \end{subfigure}
     \vskip\baselineskip
     
                \begin{subfigure}[b]{0.475\textwidth} 
                \centering
\subcaptionbox{First Iteration \label{fig:step_3}}{
\begin{tikzpicture}[scale = 0.75]
\tikzset{vertex/.style = {shape=circle,draw,minimum size=1em}}
\node[vertex]  (1) at (0,0) {1};
\node[vertex]  (2) at (2,1.25) {2};
\node[vertex]  (3) at (4,1.25) {3};
\node[vertex]  (4) at (2,-1.25) {4};
\node[vertex]  (5) at (4,-1.25) {5};
\node[vertex]  (6) at (6,0) {6};
\draw[-latex,ultra thick](1)--(2) {};
\draw[-latex,ultra thick,dashed](1)--(4) node[midway, xshift=-1mm, yshift = 1mm, fill = white] {2};
\draw[-latex,ultra thick](2)--(3) {};
\draw[-latex,ultra thick](2)--(4) {};
\draw[-latex,ultra thick](3)--(5) {};
\draw[-latex,ultra thick](3)--(6) {};
\draw[-latex,ultra thick](4)--(3) {};
\draw[-latex,ultra thick,dashed](4)--(5) node[midway, xshift=-1mm, yshift = 0mm, fill = white] {2};
\draw[-latex,ultra thick,dashed](5)--(6) node[midway, xshift=-1mm, yshift = 0mm, fill = white] {2};
\end{tikzpicture}}
        \end{subfigure}
        \hfill
        \begin{subfigure}[b]{0.475\textwidth}
            \centering
\subcaptionbox{Second Iteration \label{fig:step_1}}{
\begin{tikzpicture}[scale = 0.75]
\tikzset{vertex/.style = {shape=circle,draw,minimum size=1em}}
\node[vertex]  (1) at (0,0) {1};
\node[vertex]  (2) at (2,1.25) {2};
\node[vertex]  (3) at (4,1.25) {3};
\node[vertex]  (4) at (2,-1.25) {4};
\node[vertex]  (5) at (4,-1.25) {5};
\node[vertex]  (6) at (6,0) {6};
\draw[-latex,ultra thick,dashed](1)--(2) node[midway, xshift=0mm, yshift = 0mm, fill = white] {3};
\draw[-latex,ultra thick](1)--(4) {};
\draw[-latex,ultra thick,dashed](2)--(3) node[midway, xshift=-1mm, yshift = 0mm, fill = white] {3};
\draw[-latex,ultra thick](2)--(4) {};
\draw[-latex,ultra thick](3)--(5) {};
\draw[-latex,ultra thick,dashed](3)--(6) node[midway, xshift=0mm, yshift = 0mm, fill = white] {3};
\draw[-latex,ultra thick](4)--(3)  {};
\draw[-latex,ultra thick](4)--(5)  {};
\draw[-latex,ultra thick](5)--(6)  {};
\end{tikzpicture}}
        \end{subfigure}
        \vskip\baselineskip
 \begin{subfigure}[b]{0.475\textwidth}  
            \centering 
\subcaptionbox{Third Iteration \label{fig:step_2}}{
\begin{tikzpicture}[scale = 0.75]
\tikzset{vertex/.style = {shape=circle,draw,minimum size=1em}}
\node[vertex]  (1) at (0,0) {1};
\node[vertex]  (2) at (2,1.25) {2};
\node[vertex]  (3) at (4,1.25) {3};
\node[vertex]  (4) at (2,-1.25) {4};
\node[vertex]  (5) at (4,-1.25) {5};
\node[vertex]  (6) at (6,0) {6};
\draw[-latex,ultra thick](1)--(2) {};
\draw[-latex,ultra thick,dashed](1)--(4) node[midway, xshift=0mm, yshift = 0mm, fill = white] {4};
\draw[-latex,ultra thick](2)--(3) {};
\draw[-latex,ultra thick](2)--(4) {};
\draw[-latex,ultra thick](3)--(5) {};
\draw[-latex,ultra thick,dashed](3)--(6) node[midway, xshift=0mm, yshift = 0mm, fill = white] {4};
\draw[-latex,ultra thick,dashed](4)--(3) node[midway, xshift=0mm, yshift = 0mm, fill = white] {4};
\draw[-latex,ultra thick](4)--(5) {};
\draw[-latex,ultra thick](5)--(6) {};
\end{tikzpicture}}
        \end{subfigure}
        \hfill
        \begin{subfigure}[b]{0.475\textwidth}   
            \centering 
\subcaptionbox{Final Iteration \label{fig:step_4}}{
\begin{tikzpicture}[scale = 0.75]
\tikzset{vertex/.style = {shape=circle,draw,minimum size=1em}}
\node[vertex]  (1) at (0,0) {1};
\node[vertex]  (2) at (2,1.25) {2};
\node[vertex]  (3) at (4,1.25) {3};
\node[vertex]  (4) at (2,-1.25) {4};
\node[vertex]  (5) at (4,-1.25) {5};
\node[vertex]  (6) at (6,0) {6};
\draw[-latex,ultra thick,dashed](1)--(2) node[midway, xshift=-1mm, yshift = -1mm, fill = white] {2};
\draw[-latex,ultra thick](1)--(4) {};
\draw[-latex,ultra thick](2)--(3) {};
\draw[-latex,ultra thick,dashed](2)--(4) node[midway, xshift=0mm, yshift = 0mm, fill = white] {2};
\draw[-latex,ultra thick,dashed](3)--(5) node[midway, xshift=0mm, yshift = 0mm, fill = white] {2};
\draw[-latex,ultra thick](3)--(6) {};
\draw[-latex,ultra thick,dashed](4)--(3) node[midway, xshift=0mm, yshift = 0mm, fill = white] {2};
\draw[-latex,ultra thick](4)--(5)  {};
\draw[-latex,ultra thick,dashed](5)--(6) node[midway, xshift=-1mm, yshift = -1mm, fill = white] {2};
\end{tikzpicture}}
        \end{subfigure}
        \caption{}
        {\small Iterations of a steepest-ascent circuit augmentation scheme for a small example.
        The run replicates the steps of the Shortest Augmenting Path Algorithm.} 
        \label{fig:sapa_saa_its}
    \end{figure}

\subsection{Hungarian Method}\label{sec:hm}

The Hungarian Method (HM) is a classical algorithm for solving the assignment problem. See Algorithm \ref{alg:hung} for a description in pseudocode in the original matrix notation \cite{m-57}; the algorithm also lends itself to a natural graph-theoretical interpretation. We begin by relating the HM to pseudoflows and the pseudoflow polyhedron. 

Recall that the assignment problem is a special case of min-cost flow problems in that it runs on a bipartite network $N = L \mathbin{\dot{\cup}} T$ and every node $\ell\in L$ is a supply node with $b_\ell = 1$ and every node $t\in T$ is a demand node with $b_t = -1$. We interpret the set $L$ as people and the set $T$ as tasks. The goal is to assign each person to precisely one task and vice versa, which can formally be stated as
\begin{align*}
    \sum_{\ell \in L} x_{\ell t}  &= 1 \quad \forall t \in T\\
    \sum_{t \in T} x_{\ell t}  &= 1 \quad \forall \ell \in L \tag{assignment}\\
    x_{\ell t} &\geq 0 \quad \forall \ell \in L, \forall t \in T.
\end{align*}

The node sets $L$ and $T$ are connected in a complete bipartite network with directed arcs from $L$ to $T$. Each $x_{\ell t}$ represents flow along the arc $(\ell,t)$. The flow along each arc is implicitly bounded above by $1$ due to the main constraints. In the transfer to (\ref{pseudo_poly_form}), we may set ${\bf u}=\bm{1_m}$. The HM first greedily constructs a partial matching, where only a proper subset of people and tasks are assigned, and then iteratively finds partial matchings that include more people and tasks until all are matched. These partial matchings are readily interpreted as a $0,1$-pseudoflow: they satisfy the capacity constraints $x_{\ell t}\geq 0$, but not the flow-balance constraints for unmatched nodes. 

\begin{algorithm}
       \caption{Hungarian Method}\label{alg:hung}
\begin{algorithmic}[1] 
       \STATE \mbox{\textbf{Input:} a matrix $C$ representing the costs of each assignment for an assignment problem}
       \STATE \mbox{\textbf{Begin:}}
       \STATE Subtract the smallest element from each row of $C$
       \STATE Subtract the smallest element from each column of $C$
       \FOR{\mbox{each zero $z$ in $C'$}}
           \IF{\mbox{there are no other starred zeros in the same column and row as $z$}}
                \STATE \mbox{Star $z$}
           \ENDIF
       \ENDFOR
       \STATE Cover every column of $C'$ that contains a starred zero
       \WHILE{there is an uncovered zero}
       \FOR{\mbox{each uncovered zero $z'$ in $C'$}}
           \STATE \mbox{Prime $z'$}
           \IF{\mbox{$z'$ is in a row without a starred zero $z^*$}}
               \STATE Create an inclusion-maximal alternating path of starred and primed zeros, beginning with $z'$ ending with a primed zero 
               \STATE Unstar each starred zero in the alternating path
               \STATE Turn each primed zero into a starred zero
               \STATE Erase all primes in $C'$
               \STATE Uncover every row
               \STATE Cover every column with a starred zero
               \IF{all columns are covered}
                   \STATE Terminate
               \ENDIF
           \ELSE
               \STATE Cover the row of $z'$ and uncover the column of $z^*$
           \ENDIF
        \ENDFOR
           \IF{there are no more uncovered zeros}
               \STATE \mbox{Set $h$ as the smallest uncovered element in $C'$}
               \STATE \mbox{Add $h$ to each covered row}
               \STATE \mbox{Subtract $h$ from each uncovered column}
           \ENDIF
       \ENDWHILE
       \STATE \RETURN{} optimal assignments via the locations of the starred zeros in $C'$
\end{algorithmic}
\end{algorithm}

Before we provide a brief overview of the steps in the HM, let us comment on two of its main design principles: the effect of row and column subtractions and the role of the zeros in a cost matrix. The algorithm takes a square matrix as input. The entries in this matrix correspond to the costs $c_{\ell t}$ of assigning each person $\ell$ to each task $t$, i.e., the costs of a unit flow $x_{\ell t}$ along arc $(\ell,t)$. As the final goal is to assign every person to a task and vice versa, there will be a minimum cost incurred for each assignment of a person or task: the cheapest cost of an arc incident to the corresponding node. Subtracting the smallest value from a row or a column accounts for this `guaranteed' cost. Let $C=(c_{\ell t})$ denote the matrix of original costs and $C'=(c'_{\ell t})$ denote a matrix of costs after a sequence of row and column subtractions or additions. It is convenient to denote the total change to a row $\ell$ or column $t$ as a potential $-\pi(\ell)$ and $\pi(t)$ for the associated nodes. Then $c'_{\ell t}$ represents a current reduced cost $c'_{\ell t}=c^\pi_{\ell t} = c_{\ell t} - \pi(\ell) + \pi(t)$. Let $\mathcal{M}(L',T')$ denote the set of incidence vectors $x=(x_{\ell t})$ of perfect matchings between the set of people $L'\subset L$ and set of tasks $T'\subset T$. Then $$\arg\min_{x \in \mathcal{M}(L',T')} \sum\limits_{\ell \in L}\sum\limits_{t \in T} c'_{\ell t} \cdot x_{\ell t}  = \arg\min_{x \in \mathcal{M}(L',T')} \sum\limits_{\ell \in L}\sum\limits_{t \in T} c_{\ell t} \cdot x_{\ell t},$$ as the difference between the two sides of the equation is the fixed value $\sum_{\ell \in L'} -\pi(\ell)+\sum_{t \in T'} \pi(t)$. Essentially, an optimal perfect matching between given sets $L'\subset L$ and $T'\subset T$ is not affected by the row and column operations. This is known as the Kuhn-Munkres Theorem and the key argument for correctness of the HM \cite{h-55,m-57}. 

The algorithm maintains reduced costs $c^\pi_{\ell t}\geq 0$ for all arcs at all times. The goal of the row and column operations is to create a network of $0$-cost arcs in the matrix that allow for the identification of an optimal perfect matching of larger and larger subset $L',T'$ of $L,T$. 
There are two ways in which 0-cost arcs are marked up and used: the {\em starring} of a zero translates to adding the associated arc to the current partial matching, while {\em priming} a zero marks the associated arc as a possible option for an augmentation of the current partial matching while guaranteeing optimal cost for the new, larger sets $L'$ and $T'$. 

Let us now take a closer look at the steps of the HM. The algorithm starts by subtracting the smallest value from each row and then from each column of the matrix (Lines 3-4). This corresponds to `taking out' a guaranteed minimum cost incurred in any partial or complete matching; the corresponding minimal costs become zero. Each row and column now contains at least one zero. Next, the algorithm greedily chooses from the 0-cost arcs to assign people to tasks (Lines 5-8). If this greedy initial partial matching is a perfect matching (Lines 10-11), we are done. Otherwise, the algorithm must add people and tasks that have not been assigned to the partial matching which signals the start of the main loop (Lines 11-33).

The goal of the main loop is to iteratively increase the size of the current partial matching. In each iteration, the algorithm checks whether there exists an augmenting path in the current network of 0-cost arcs (Lines 12-14, 24-26). If one is found, the partial matching is updated based on this augmenting path (Lines 15-20). While not all people and tasks are matched yet (Lines 21-23), and the current network of 0-cost arcs does not allow for an increase in the size of the current partial matching, a new 0-cost arc must be created (Lines 28-32). Then the search for an augmenting path is repeated in the now larger network of 0-cost arcs. It continues in this fashion until all people are assigned a task, at which point the algorithm terminates and outputs the final perfect matching.
 
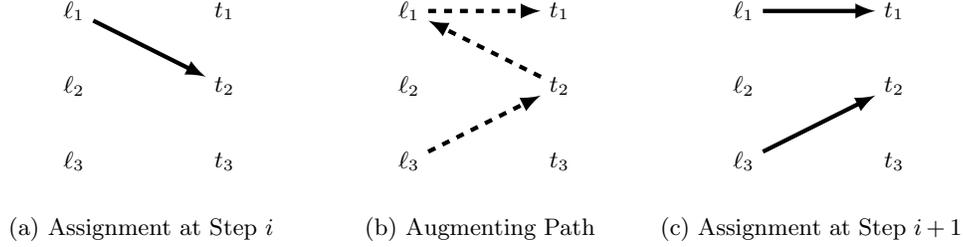
\begin{figure}
     \centering
     \begin{subfigure}[b]{0.3\textwidth}
         \centering
         \subcaptionbox{Assignment at Step $i$ \label{fig:init_assignment}}{
\begin{tikzpicture}[scale = 1]
\node  (p1) at (0,2) {$\ell_1$};
\node  (p2) at (0,1) {$\ell_2$};
\node  (p3) at (0,0) {$\ell_3$};
\node  (j1) at (2,2) {$t_1$};
\node  (j2) at (2,1) {$t_2$};
\node  (j3) at (2,0) {$t_3$};
\draw[-latex,ultra thick] (p1)--(j2) {};
\useasboundingbox  (-1,-0.5) rectangle (3,2.5);
\end{tikzpicture}}
     \end{subfigure}
     \begin{subfigure}[b]{0.3\textwidth}
         \centering
         \subcaptionbox{Augmenting Path}{
\begin{tikzpicture}[scale = 1]
\node  (p1) at (0,2) {$\ell_1$};
\node  (p2) at (0,1) {$\ell_2$};
\node  (p3) at (0,0) {$\ell_3$};
\node  (j1) at (2,2) {$t_1$};
\node  (j2) at (2,1) {$t_2$};
\node  (j3) at (2,0) {$t_3$};
\draw[-latex,ultra thick,dashed] (p3)--(j2) {};
\draw[-latex,ultra thick,dashed] (j2)--(p1) {};
\draw[-latex,ultra thick,dashed] (p1)--(j1) {};
\useasboundingbox  (-1,-0.5) rectangle (3,2.5);
\end{tikzpicture}}
     \end{subfigure}
     \begin{subfigure}[b]{0.3\textwidth}
         \centering
         \subcaptionbox{Assignment at Step $i+1$ \label{fig:end_assignment}}{
\begin{tikzpicture}[scale = 1]
\node  (p1) at (0,2) {$\ell_1$};
\node  (p2) at (0,1) {$\ell_2$};
\node  (p3) at (0,0) {$\ell_3$};
\node  (j1) at (2,2) {$t_1$};
\node  (j2) at (2,1) {$t_2$};
\node  (j3) at (2,0) {$t_3$};
\draw[-latex,ultra thick] (p1)--(j1) {};
\draw[-latex,ultra thick] (p3)--(j2) {};
\useasboundingbox  (-1,-0.5) rectangle (3,2.5);
\end{tikzpicture}}
     \end{subfigure}
        \caption{An example for an augmentation step in the Hungarian Method.}
        \label{fig:hm_step}
\end{figure}

We now turn to the results of this section. We begin by showing that the HM performs a circuit walk over a special face of (\ref{pseudo_poly_form}) (Lemma \ref{lem:hm_circ_walk}). Then we show that every circuit walk for the assignment problem over that face is an edge walk (Theorem \ref{thm:hm_edge_walk}). Next, we provide a description of a circuit augmentation scheme which replicates a run of the HM (Theorem \ref{thm:hm_dantzig}) on the same face of (\ref{pseudo_poly_form}). Finally, we conclude with the observation that a primal Simplex run which uses the largest-coefficient rule is an alternative method for replicating a run of the HM (Corollary \ref{cor:hm_dantzig}). 

Due to the fact that the assignment problem is a special case of the min-cost flow problem, some arguments from Section \ref{sec:sspa} are helpful again. We begin by showing that the HM performs a circuit walk over a certain face of (\ref{pseudo_poly_form}).

\begin{lemma}\label{lem:hm_circ_walk}
    The Hungarian Method performs a circuit walk over the face, $F$, of (\ref{pseudo_poly_form}) where $s^+_{\ell} = 0$ for all $\ell \in L$ and $s^-_t = 0$ for all $t \in T$. 
\end{lemma}

\begin{proof}
We restrict to the face $F$ of (\ref{pseudo_poly_form}) where $s^+_{\ell} = 0$ for all $\ell \in L$ and $s^-_t = 0$ for all $t \in T$. The HM begins at the $0,1$-vertex $\bm{y_0} = (\bm{0_{m}}, \bm{s^+}, \bm{s^-})$ with $s_{t}^+ = -b_{t}=1$ for all $t \in T$ and $s_{\ell}^- = b_{\ell}=1$ for all $\ell \in L$ (Lemma \ref{lem:feas_vert}). Note that $\bm{y_0}$ lies in $F$. The algorithm then greedily assigns people to tasks (Lines 5-8). The assignment between a given person $\ell$ and task $t$ in each step behaves like forming a one-arc path between a supply node $\ell$ and a demand node $t$ in the SSPA. Therefore, several arguments for path circuits from Lemma \ref{lem:sspa_circ_walk} also hold here: once the assignment ($\ell$, $t$) is chosen, flow is sent along $x_{\ell t}$. This corresponds to the path circuit $\bm{g} = (\bm{x_{\ell t}},-\bm{e_{t}},-\bm{e_{\ell}})$ due to the restriction to $F$. After completing the circuit step $\bm{y_1} = \bm{y_0} + \delta \bm{g}$, we arrive at the point $\bm{y_1} = (\bm{x_{\ell t}},\bm{s^+}-\bm{e_{t}}, \bm{s^-}-\bm{e_{\ell}}) \in F$. The step length is $\delta=1$, and maximal, due to the upper bound $1$ on flow along $x_{\ell t}$ and reducing slack variables $s^+_t,s^-_\ell$ from $1$ to $0$; $\bm{y_1}$ again is a $0,1$-vector. Therefore, $\bm{y_1}$ satisfies the same properties used for $\bm{y_0}$. The arguments hold for subsequent circuit steps after $\bm{y_1}$, corresponding to the remainder of the greedy assignments. 

The main loop (Lines 11-33) begins with a partial matching $M_k$ corresponding to a $0,1$-vector $\bm{y_k} =(\bm{x_{k}},\bm{s^+_{k}},\bm{s^-_{k}})$ where $\bm{x_{k}} = \sum_{(\ell, t) \in M_k}\bm{e_{\ell t}}$, $\bm{s^+_{k}} = s^+-\sum_{\{t:(\ell,t) \in M_k\}} \bm{e_{t}}$, and $\bm{s^-_{k}} = s^--\sum_{\{\ell:(\ell,t) \in M_k\}} \bm{e_{\ell}}$. In each iteration, the current partial assignment and corresponding $\bm{y_k}$ can only be changed through an augmenting path (Lines 14-23). Such a path $H_{\ell t}$ must connect a person $\ell$ and a task $t$ that are not in the node set of $M_k$. Thus, $s^+_t=s^-_\ell=1$ in $\bm{y_k}$. Among the path circuits corresponding to $H_{\ell t}$, only $\bm{g} = (\bm{H_{\ell t}},-\bm{e_{t}},-\bm{e_{\ell}})$ remains in $F$. The circuit step $\bm{y_{k+1}} = \bm{y_k}+\delta \bm{g} \in F$ is of maximal length $\delta=1$ and one arrives at the $0,1$-vector $\bm{y_{k+1}} = (\bm{x_{M_k}}+\bm{H_{\ell t}},\bm{s^+_{M_k}}-\bm{e_{t}},\bm{s^-_{M_k}}-\bm{e_{\ell}})=(\bm{x_{M_{k+1}}},\bm{s^+_{M_{k+1}}},\bm{s^-_{M_{k+1}}}) \in F$ corresponding to the new partial matching $M_{k+1}$. 

The arguments for the circuit step from  $\bm{y_k}$ to $\bm{y_{k+1}}$ also hold for later iterations. This proves the claim.
\end{proof}

Our next goal is to show that the HM performs an edge walk. We prove a slightly stronger statement: in the face $F$ from Lemma \ref{lem:hm_circ_walk}, every circuit walk is an edge walk.

\begin{theorem}\label{thm:hm_edge_walk}
 For the pseudoflow polyhedron of the assignment problem restricted to the face, $F$, where $s^+_{\ell} = 0$ for all $\ell \in L$ and $s^-_t = 0$ for all $t \in T$, every circuit walk is an edge walk. 
\end{theorem}

\begin{proof}
First, note that the face $F$ is bounded (in the unbounded pseudoflow polyhedron): 
the upper capacities of $1$ for all $x_{\ell t}$ and the flow balance constraints, coupled with the restriction to $F$, give an implicit upper bound on the unrestricted slack arcs $s^-_{\ell}$ and $s^+_t$ for all $\ell \in L, t \in T$. This allows us to prove the claim by showing that any circuit step starting at a vertex $\bf{y_k}$ of face $F$ follows an edge direction to an adjacent vertex $\bf{y_{k+1}}$. Further, steps along trivial circuits are not feasible in $F$, and so it suffices to consider cycle circuits and path circuits.

By total unimodularity of the node-arc incidence matrix of the underlying pseudoflow network and by $0\leq {\bf x} \leq 1$, the face $F$ has only integral vertices $\bm{y_k} =(\bm{x_k},\bm{s^+_k},\bm{s^-_k})$ with $0,1$-vector $\bm{x_k}$ corresponding to a selection of arcs. In particular, the selection of arcs must correspond to a complete or partial matching since, when restricted to $F$, all arcs for $\ell \in L$ and $t \in T$ in the pseudoflow network are outgoing and ingoing, respectively. We consider a step along a feasible cycle or path circuit starting at $\bm{y_k}$. Due to $\bm{x_k} \in \{0,1\}$, it is guaranteed to have step length $\delta =1$ and arrives at a  $\bm{y_{k+1}} =(\bm{x_{k+1}},\bm{s^+_{k+1}},\bm{s^-_{k+1}}) \in F$, where $\bm{x_{k+1}}$ corresponds to a new partial matching. This $\bm{y_{k+1}}$ is uniquely maximal over $F$ with respect to objective function $\bm{c}$, where 
\[ c_{ij} = 
    \begin{cases} 
      1 &  \text{when } (x_{k+1})_{ij} = 1\\
      -1 & \text{when } (x_{k+1})_{ij} = 0\\
      0 & \text{for all slack arcs}
   \end{cases}
\]
Thus $\bm{y_{k+1}}$ is another vertex of $F$ (that differs from $\bm{y_{k}}$ by a single circuit). The objective function $\bm{c}$ can also be used to observe that all vertices of $F$ are a complete or partial matching. 

It remains to consider pairs of vertices $\bm{y_1}=(\bm{x_1},\bm{s^+_1},\bm{s^-_1})$ and $\bm{y_2}=(\bm{x_2},\bm{s^+_2},\bm{s^-_2})$ of $F$ that differ by a single cycle or path circuit, with the aim of constructing an objective function $\bm{c}$ such that $\bm{c^Ty_1} = \bm{c^Ty_2} > \bm{c^Ty_3}$ for all vertices $\bm{y_3} \neq \bm{y_1},\bm{y_2}$ of $F$. This will imply that they are adjacent. 

First, let $\bm{y_2}$ differ from $\bm{y_1}$ by a path circuit $\bm{g} = (\bm{H_{\ell t}},-\bm{e_{t}},-\bm{e_{\ell}})$; recall that the other path circuits for $H_{\ell t}$ from Corollary \ref{cor:final_circ_char} are not feasible in $F$. Further, let $H_{\ell t}$ be an augmenting (alternating) path, c.f. Figure \ref{fig:hm_step}, where $k$ is the number of people in $\bm{y_1}$ that are reassigned and $q$ the number of assignments shared between $\bm{y_1}$ and $\bm{y_2}$. Thus, $\bm{y_2}$ has one more assignment compared to $\bm{y_1}$. 
Define $\bm{c}$ to be 
\[ c_{ij} = 
    \begin{cases} 
      1 &  \text{when } (x_1)_{ij} = 1 \\
      -1 & \text{when } (x_1)_{ij} = (x_2)_{ij} = 0\\
      \frac{k}{k+1} & \text{when } (x_2)_{ij} = 1 \neq (x_1)_{ij}\\
      0 & \text{for all slack arcs}
   \end{cases}
\]
The objective function value for $\bm{y_1}, \bm{y_2}$ then becomes
\begin{align*}
    \bm{c}^T\bm{y_1} = q+ \sum^{k}_{m=1}1 = q+ k = q+ \sum^{k+1}_{m=1}\frac{k}{k+1} = \bm{c}^T\bm{y_2}
\end{align*} 
and $\bm{c^Ty_3} < q+k =\bm{c^Ty_1} = \bm{c^Ty_2}$ when $\bm{y_3}$ shares no assignments with $\bm{y_1}$ or $\bm{y_2}$ or when it is a subset of assignments from $\bm{y_1}$ or $\bm{y_2}$. If $\bm{y_3}$ contains a combination of assignments from $\bm{y_1}$ and $\bm{y_2}$, then the cost of the matching in $\bm{y_3}$ must be strictly less than that of $\bm{y_1}$ and $\bm{y_2}$. This is because nodes contained in the matchings of $\bm{y_1}$ and $\bm{y_2}$ cannot have multiple assignments and $\bm{y_3} \neq \bm{y_1},\bm{y_2}$. Therefore, $\bm{c}^T\bm{y_3} < \bm{c}^T\bm{y_1} = \bm{c}^T\bm{y_2}$ for all vertices $\bm{y_3} \neq \bm{y_1}, \bm{y_2}$.

Similar arguments hold when $\bm{y_1}$ and $\bm{y_2}$ differ by a cycle circuit or a path circuit where $H_{\ell t}$ is not an augmenting path. In these cases, the number of assignments in $\bm{y_1}$ and $\bm{y_2}$ is the same. We define $\bm{c}$ to be
\[ c_{ij} = 
    \begin{cases} 
      1 &  \text{when } (x_1)_{ij} = 1 \text{ or } (x_2)_{ij} = 1\\
      -1 & \text{when } (x_1)_{ij} = (x_2)_{ij} = 0\\
      0 & \text{for all slack arcs}
   \end{cases}
\]
which immediately gives $\bm{c}^T\bm{y_1} = \bm{c}^T\bm{y_2}$. 
Again, we see that $\bm{c}^T\bm{y_3} < \bm{c}^T\bm{y_1} = \bm{c}^T\bm{y_2}$ for all vertices $\bm{y_3} \neq \bm{y_1}, \bm{y_2}$.

We have shown that circuit steps starting at a vertex of $F$ lead to another vertex of $F$, and that vertices which differ by a path or cycle circuit are adjacent. Thus, all circuit walks on $F$ are edge walks. 
\end{proof}

From Lemma \ref{lem:hm_circ_walk}, we know that the HM performs a circuit walk in the face $F$ of (\ref{pseudo_poly_form}) for the assignment problem. We also know from Theorem \ref{thm:hm_edge_walk} that all circuit walks in $F$ are, in fact, edge walks. Together, this allows us to conclude that the HM performs an edge walk. 

\begin{corollary}\label{cor:hm_edge_walk}
    The Hungarian Method performs an edge walk on the face of (\ref{pseudo_poly_form}) where $s^+_{\ell} = 0$ for all $\ell \in L$ and $s^-_t = 0$ for all $ t \in T$.
\end{corollary}

We next show that a run of the HM can be replicated using a Dantzig circuit augmentation scheme. To this end, we can make use of some of our observations for the SSPA. The assignment problem is a special case of a min-cost flow problem. In Theorem \ref{thm:sspa_dantzig}, we explained the behavior of a circuit augmentation scheme using Dantzig's rule for an objective function that penalizes the use of slack arcs: it corresponds to a sequence of augmenting (alternating) paths as in the SSPA. The SSPA and the circuit augmentation scheme terminate with the correct solution once there are no more excess and deficit nodes. We begin the proof by transferring these arguments to the HM.

A difference between the HM and a special case of the SSPA lies in the HM precisely specifying the new pair of matched nodes. In the second part of the proof, we explain how to adjust an objective function to arrive at the same selection of paths. The augmentation scheme encompasses both pre-processing (Lines 5-8) and the main loop (Lines 11-33) of the HM.

\begin{theorem}\label{thm:hm_dantzig}
    A circuit augmentation scheme over (\ref{pseudo_poly_form}) using Dantzig’s rule replicates a run of the Hungarian Method.
\end{theorem}

\begin{proof}
We prove the claim by providing an objective function $\bm{c} = (\bm{c_x},\bm{c_{s^+}},\bm{c_{s^-}}) \in \R^{m+2n}$ such that the Dantzig circuits $\bm{g}$ correspond to the path circuits from Lemma \ref{lem:hm_circ_walk} in the face $F$. The vector $\bm{c_x} \in \R^{m}$ corresponds to the (non-negative) costs of the arcs from the original network. As we will see, it suffices to provide appropriate costs $\bm{c_{s^+}},\bm{c_{s^-}}$ for the slack variables.

We begin with a transfer of arguments for the SSPA that also apply to the HM. We discourage flow on slack arcs through the use of a sufficiently large cost $M \gg D:= 1 + \sum_{(i,j) \in A} c_{ij}$ and set $\bm{c_{s^+}} = \bm{c_{s^-}} = \bm{M_n} = M\cdot \bm{1_{n}}$. Using $\bm{c} = (\bm{c_x},\bm{M_n},\bm{M_n})$ we reward assignments made using $x$-arcs and penalize the use of slack arcs. Let us consider the value of $-\bm{c}^T\bm{g}$ for each type of circuit from Corollary \ref{cor:final_circ_char}. Note that $-M \ll -D < \bm{c_x^Tg_x} < D \ll M$ for any circuit $\bm{g}$, where $\bm{g}=(\bm{g_x},\bm{g_{s^+}},\bm{g_{s^-}})$ and $\bm{g_x}$ are the indices of the circuit $g$ that corresponds to $x$-arcs. This again allows a categorization of the circuits by the slack contribution $s_g=\bm{M_{n}^T\bm{g_{s^+}}}+\bm{M_{n}^T\bm{g_{s^-}}}$: trivial circuits have $s_g=2M$ or $s_g=-2M$; cycle circuits have $s_g = 0$; path circuits have $s_g=-2M$, $s_g=0$, or $s_g=2M$. If there exists a feasible circuit with slack contribution $-2M$, which corresponds to reducing flow along two slack arcs, a maximum value of $-\bm{c^Tg}$ can only be achieved for such a circuit. Therefore $\bm{g}$ can only be a path circuit or a trivial circuit.

The circuit augmentation scheme starts at the vertex $\bm{y_0}=(\bm{x_0},\bm{s^+},\bm{s^-})$ in $F$ with $\bm{x_0}=\bm{0_m}$. Because at least one of each pair $s^{\pm}_i$ is equal to zero, trivial circuits are not feasible at $\bm{y_0}$. For any pair of person and task ($\ell, t$), there exists a path circuit $\bm{g}=(\bm{H_{\ell t}},-\bm{e_{t}},-\bm{e_{\ell}})$ that reduces flow along two slack arcs. Dantzig's rule chooses a best path circuit $\bm{g}$ among these and performs a maximal circuit step to $\bm{y_{1}}=\bm{y_0} + \delta \bm{g}$ in $F$. For $\bm{y_{1}}$ again at least one of each pair $s_i^\pm$ is equal to zero and a non-zero slack variable corresponds to the remaining excess or deficit one at the node. This implies that the above arguments can be repeated for $\bm{y_{k+1}}$ in place of $\bm{y_k}$: trivial circuits are not feasible; while there remain unmatched people and tasks, a path circuit that reduces flow along two non-zero slack arcs is chosen; and the new $\bm{y_{k+2}}$ satisfies the same properties again. This augmentation scheme replicates a run of the SSPA on the assignment problem, and terminates with the correct solution (Theorem \ref{thm:sspa_dantzig}).

A key difference between the HM and the SSPA lies in the choice of pairs $(\ell,t)$ of person $\ell$ and task $t$ in the HM. We complete the proof by adjusting the objective function $\bm{c}$ so that the order of augmenting paths is precisely that of the given run of the HM.
In the pre-processing stage (Lines 5-8) of the HM, an initial partial matching is constructed greedily. Without loss of generality, and for the sake of simple indexing, consider a run in which a partial matching of $k-1$ arcs $(\ell_1,t_1),\dots,(\ell_{k-1},t_{k-1})$ is created in that order. Let $L',T'$ denote the corresponding sets of nodes. Each of the arcs corresponds to an iteration based on an augmenting path with just a single $x$-arc.

The selection of arcs for this initial partial matching comes from $0$-cost arcs after the initial row and column subtractions; i.e.,  $c^\pi_{\ell_i t_i}=c_{\ell_i t_i} - \pi(\ell_i) + \pi(t_i) = 0$. As $c^\pi_{\ell t}\geq 0$ for all arcs, in each iteration, the one-$x$-arc path using $(\ell_i,t_i)$ is a shortest path among all paths from $\ell_i$ to $t_i$ in the current residual network for the original costs, too. To guarantee that a path between $\ell_i$ and $t_i$ is chosen in iteration $i$ of a circuit augmentation scheme using Dantzig's rule, one can update $\bm{c} = (\bm{c_x},\bm{M_n},\bm{M_n})$ by using
$$(\bm{c_x},\bm{M_n},\bm{M_n}) + 
\left(\bm{0}, -2D \cdot \sum_{t_i \in T'} i \bm{e_{t_i}}, -2D \cdot \sum_{\ell_i \in L'} i \bm{e_{\ell_i}}\right). $$
Note that the penalty $M\gg D$ decreases by $2D$ for each subsequent matched pair of nodes.

The main loop (Lines 11-33) of the HM begins with a partial matching $\bm{y_k}=(\bm{x_{k}},\bm{s^+_{k}},\bm{s^-_{k}})$ in $F$. The current network of $0$-cost arcs is updated until an augmenting path $H_{\ell_k t_k}$ of only $0$-cost arcs between $(\ell_k,t_k)$ is found. As $c^\pi_{\ell t}\geq 0$, the reduced cost of this path again is optimal in the current residual network. As 
$$0=\sum_{(i,j) \in H_{\ell_k t_k}} c^{\pi}_{ij} = \sum_{(i,j) \in H_{\ell_k t_k}} (c_{ij} - \pi(i) + \pi(j)) = \left(\sum_{(i,j) \in H_{\ell_k t_k}} c_{ij}\right) - \pi(\ell_k) + \pi(t_k),$$
the path $H_{\ell_k t_k}$ again also is a shortest path between $\ell_k$ and $t_k$ in the current residual network for the original costs, too. The same holds for pairs of nodes connected by augmenting paths in later iterations of the main loop. For a simple notation, let  the main loop in the run form augmenting paths between pairs $(\ell_k,t_k),(\ell_{k+1},t_{k+1}),\dots,(\ell_{|L|},t_{|T|})$. To guarantee that Dantzig's rule chooses a path between these pairs of nodes in each iteration over a possibly shorter path between a different pair (which the more general SSPA could do), update $\bm{c}$ as before by setting
$$(\bm{c_x},\bm{M_n},\bm{M_n}) + 
\left(\bm{0}, -2D \cdot \sum_{t_i \in T\backslash T'} i \bm{e_{t_i}}, -2D \cdot \sum_{\ell_i \in L \backslash L'} i \bm{e_{\ell_i}}\right). $$
Note that $i\geq k+1$ for $t_i \in T\backslash T'$. This allows one to represent the overall update of $\bm{c}$ as 
$$\bm{c} \leftarrow (\bm{c_x},\bm{M_n},\bm{M_n}) + 
\left(\bm{0}, -2D \cdot \sum_{t_i \in T} i \bm{e_{t_i}}, -2D \cdot \sum_{\ell_i \in L} i \bm{e_{\ell_i}}\right). $$
 Finally, unlike for the more general SSPA, note that each $\ell \in L$ and each $t \in T$ appear as start and end of an augmenting path in the run exactly once. After $\ell \in L$ and $t \in T$ are used in this way, they remain in all partial matchings for all later iterations, until termination. This proves the claim.
\end{proof}

The adjustment of $\bm{c}$ in the proof of Theorem \ref{thm:hm_dantzig} to obtain a specific order of augmenting paths is also possible, but not necessary, for the SSPA. It builds on observing a run of the HM first. When using the first objective function in the proof, one replicates a run of the SSPA for an assignment problem; this can, of course, itself already be seen as a replication of a generalized version of the HM, too.

Combining Theorems \ref{thm:hm_edge_walk} and \ref{thm:hm_dantzig}, we know that a Dantzig augmentation scheme produces an edge walk on (\ref{pseudo_poly_form}) and replicates a run of the HM. In particular, recall from Theorem \ref{thm:hm_edge_walk} that in the face $F$ of (\ref{pseudo_poly_form}) where $s^+_{\ell} = 0$ for all $\ell \in L$ and $s^-_t = 0$ for all $t \in T$, all circuit walks are edge walks. Further, the circuits only have entries $\pm 1, 0$ (Corollary \ref{cor:final_circ_char}) and the 
step length in each iteration is $1$ (Lemma \ref{lem:hm_circ_walk}). In fact, the proof of Lemma \ref{lem:hm_circ_walk} is readily generalized to see that any adjacent vertices are at max-norm distance $1$.

This leads to an immediate correspondence of Dantzig's Rule for circuits  (Definition \ref{def:dantzig}) and a largest-coefficient rule for edges over the face $F$. In each iteration, an improving edge is selected; some (entering) nonbasic variable is increased from $0$ to $1$. The associated reduced cost is equivalent to $\bm{c^Tg}$ for the corresponding circuit $\bm{g}$. Thus, the same edge walk can be produced through a primal Simplex run over $F$ using a largest-coefficient rule. We obtain a second, alternative way for replicating a run of the HM.

\begin{corollary}\label{cor:hm_dantzig}
    A primal Simplex run over (\ref{pseudo_poly_form}) using a largest-coefficient rule replicates a run of the Hungarian Method.
\end{corollary}

\section{Conclusion and Outlook}\label{sec:conclusion}
The existence of a close connection between linear programming through Simplex variants or circuit augmentation and combinatorial algorithms for min-cost flow problems is well-known, and our contributions further deepen this understanding. We introduced and studied pseudoflow polyhedra that represent the set of pseudoflows on a network. We identified their circuits and used this information to characterize a number of classical network flows algorithms as combinatorial interpretations of circuit augmentation schemes over the polyhedra. While the Successive Shortest Path Algorithm and the Generic Augmenting Path Algorithm perform general circuit walks, the Hungarian Method performs an edge walk. All of them are replicated through simple pivot rules. 

An interesting feature and strength of our results is that we arrive at {\em purely primal interpretations} of these algorithms - contrasting with classical primal-dual interpretations of the Successive Shortest Path Algorithm and the Hungarian Method in the Simplex setting. At the same time, our approach is also limited in this way -- by its design, we can only arrive at primal interpretations. A natural direction of research is to develop the tools for an interpretation of combinatorial algorithms as primal-dual circuit augmentation; the interplay between primal steps and dual updates in the algorithm in \cite{bv-19b} is a promising starting point. 

We believe that this would allow us to arrive at an interpretation of additional network flows algorithms. One example that we are specifically interested in is the Preflow-Push algorithm described in Section \ref{sec:preflow}: while we are able to show that the algorithm traces a general circuit walk over the pseudoflow polyhedron, we believe that dual information would be required to describe a simple pivot rule that replicates the walk. 

More generally, we claim that the ability to interpret algorithms as primal-dual circuit augmentation would allow us to exhibit similar connections for combinatorial algorithms outside of the setting of network flows, such as non-bipartite matching, scheduling, routing, or packing. 
A first major challenge is a characterization of the set of circuits of the polyhedra underlying these more complicated problems. Even for combinatorial optimization problems that retain a nice graph-theoretical interpretation, such as the Traveling Salesman Problem, a full characterization of the set of circuits becomes challenging \cite{kps-17}. However, many combinatorial algorithms iteratively perform conceptually simple steps, and proving that these steps correspond to circuits would already suffice to validate that a circuit walk is performed. In contrast, for a discussion of an associated pivot rule, a full characterization of the set of circuits -- or equivalent expert knowledge -- may be required.

\subsection{Preflow-Push}\label{sec:preflow}
 
The Generic Preflow-Push Algorithm (PFP) is another algorithm for solving max flow problems. A \textit{preflow} \cite{amo-93} is a function on the arc set that satisfies the capacity constraints as well as the following flow balance relaxation:
\begin{equation*}
    e(i) = \sum_{\{j:(j,i) \in A)\}}x_{ji} - \sum_{\{j:(i,j) \in A)\}} x_{ij} \geq 0 \quad \forall i \in N\setminus \{s,t\}.
\end{equation*}
Preflows are a special case of pseudoflows in that any flow imbalance incurred is an excess ($e(i) \geq 0$). Thus, we restrict to the face of (\ref{pseudo_poly_form}) where $s^-_s = 0$ and $s^+_i = 0$ for all $i \in N \setminus\{s\}$. See Algorithm \ref{alg:PFP} for a description in pseudocode \cite{amo-93}. To aid in its discussion, let us define some key terms of the PFP. An \textit{active node} is a node $i\in N \setminus\{s,t\}$ with strictly positive excess, $e(i)>0$. For a chosen active node $i$, an \textit{admissible arc} is an arc $(i,j)$ such that $d(i) = d(j)+1$, where $d(i)$ and $d(j)$ are the distance labels of node $i$ and $j$, respectively. Initially, these labels refer to the distance of a node from $t$; later these are lower bounds on the distance. When flow is sent along an admissible arc, we say the arc is \textit{saturated} if the flow, $\delta$, along the arc is its residual capacity, $r_{ij}$, and \textit{nonsaturating} otherwise.

The pre-processing stage of the PFP (Lines 3-6) begins with zero flow along all arcs (Line 3). Then the initial distance labels for nodes $i \in N\setminus\{s,t\}$ are computed (Line 4) and all arcs $(s,i)$ leaving the source node $s$ are saturated (Line 5); the corresponding nodes $i$ become active. Finally, the source $s$ receives a large distance label, distinguishing it from the exact distance labels for the other nodes (Line 6). 

Then the algorithm enters the main loop (Lines 7-15). Each iteration begins by selecting an active node $i$ (Line 8). If there exists an admissible arc for the node (Line 9), flow is sent based on the smaller of the amount of excess at the active node or the capacity of the admissible arc (Line 10), and the residual network is updated (Line 11). If no admissible arc exists, the algorithm notices that the current distance label $d(i)$ is an underestimation of the shortest distance to $t$ and increases it to the smallest value that may be the correct distance based on the labels for its neighbors (Line 13). Thus, in each iteration either flow is `pushed' towards a node closer to the target $t$ or a distance label is increased. If the label for a node becomes larger than the label of node $s$, the remaining excess cannot be sent to $t$ anymore; subsequent iterations push it back towards $s$. The algorithm terminates when no active nodes remain in the network. 

\begin{algorithm}
       \caption{Generic Preflow-Push Algorithm}\label{alg:PFP}
\begin{algorithmic}[1] 
       \STATE \mbox{\textbf{Input:} a network $G(N,A)$ with non-negative capacities $\bm{u}$}
       \STATE \mbox{\textbf{Begin:}}
       \STATE Set arcs $\bm{x} = \bm{0}$
       \STATE Compute the exact distance label $d(i)$ $\forall i \in N\setminus \{s,t\}$
       \STATE Set $x_{si} = u_{si}$ for each arc $(s,i)$ where $s$ is the source node
       \STATE Set $d(s) = n$

       \WHILE{the network contains an active node}
            \STATE Select an active node $i$
            \IF{the network contains an admissible arc $(i,j)$}
                \STATE Push $\delta = \min\{e(i), r_{ij}\}$ units of flow from node $i$ to node $j$
                \STATE Update $G(N,A)$
            \ELSE
                \STATE {Replace $d(i)$ by $\min\{d(j)+1:(i,j) \in A(i) \text{ and } r_{ij} > 0\}$}
            \ENDIF
       \ENDWHILE
       \STATE \RETURN{} max flow $\bm{x}$
\end{algorithmic}
\end{algorithm}

As with some of the algorithms in Section \ref{sec:alg_of_int}, we first show that the PFP performs a circuit walk on a face of the pseudoflow polyhedron (Lemma \ref{lem:pfp_circ_walk}). We then provide an example in which a run of the PFP performs a general (non-edge) walk (Figure \ref{fig:pfp_gen_walk}). However, unlike for the algorithms in Section \ref{sec:alg_of_int}, it will remain open what pivot rule would allow a replication of the run. 

\begin{lemma}\label{lem:pfp_circ_walk}
    The Preflow-Push Algorithm performs a circuit walk over the face $F$ of (\ref{pseudo_poly_form}) where $s^-_s = 0$ and $s^+_i = 0$ for all $i \in N \setminus \{s\}$. 
\end{lemma}

\begin{proof}
We begin at the vertex $\bm{y_0} = \bm{0_{m+2n}}$ and restrict to the face  $F$ of (\ref{pseudo_poly_form}) where $s^-_s = 0$ and $s^+_i = 0$ for all $i \in N \setminus \{s\}$. During the pre-processing stage, the $x$-arcs which emanate from $s$ are saturated. Since each arc can be saturated one at a time and be interpreted as a one-arc path, the pre-processing stage corresponds to a sequence of path circuits of the form $\bm{g} = (\bm{e_{si}},\bm{e_s},\bm{e_i})$ (Corollary \ref{cor:final_circ_char}). Among the path circuits corresponding to the one-arc path, only $\bm{g}$ is feasible at $\bm{y_0}$ (all others would have to decrease a slack variable below $0$). The circuit step $\bm{y_1} = \bm{y_0}+\delta \bm{g} \in F$ corresponding to the saturation of $(s,i)$ with capacity $u_{si}$ is of maximal length $\delta = u_{si}$ in (\ref{pseudo_poly_form}). These arguments hold for subsequent circuit steps after $\bm{y_1}$. Thus, the pre-processing stage corresponds to a circuit walk in $F$.  

The pre-processing concludes at the point $\bm{y_k} = (\bm{u^k_{si}}, \bm{u^k_s}, \bm{u^k_i})$, where $\bm{u^k_{si}}=\sum_{\{i:(s,i) \in A)\}} u_{si}\cdot \bm{e_{si}}$, $\bm{u^k_s} = \bm{e_{s}} \cdot \sum_{\{i:(s,i) \in A)\}} u_{si}$, and $\bm{u^k_i}= \sum_{\{i:(s,i) \in A)\}} u_{si}\cdot \bm{e_{i}}$ denote the corresponding sums of scaled unit vectors. Note that $\bm{u^k_s}$ remains a scaled unit vector itself.  

The main loop begins at $\bm{y_k}$. In each iteration, flow can only be changed along a single admissible arc (Lines 10-11). This change again corresponds to one of the path circuits from Corollary \ref{cor:final_circ_char}. Only two types of path circuits can appear. For arcs $(i,j)$ with $i,j\neq s$, circuits of the form $\bm{g} = (\bm{e_{ij}},\bm{0_m},\bm{e_j}-\bm{e_i})$ (or $\bm{g} = (-\bm{e_{ij}},\bm{0_m},\bm{e_i}-\bm{e_j})$ when sending flow along opposite arc $(j,i)$ in the residual network) are the only ones to satisfy $\bm{y_{k+1}} = \bm{y_k}+\delta \bm{g} \in F$. The step length $\delta$ in the PFP (Line 10) corresponds to $x_{ij}$ becoming equal to the upper capacity (or zero) or to $s^-_i$ (or $s^-_j$) becoming zero, so the step is of maximal length over (\ref{pseudo_poly_form}). 

For arcs $(i,s)$ in the residual network (which come from the pre-processing and saturation of original arcs $(s,i)$), circuits of the form $\bm{g} = (-\bm{e_{si}}, \bm{e_s},-\bm{e_i})$ are the only ones to satisfy $\bm{y_{k+1}} = \bm{y_k}+\delta \bm{g} \in F$. The step length $\delta$ in the PFP (Line 10) corresponds to the opposite flow $x_{si}$ becoming equal to zero or to $s^-_i$ becoming zero, so the step is of maximal length over (\ref{pseudo_poly_form}).

The arguments for the circuit step from $\bm{y_{k}}$ and $\bm{y_{k+1}}$ also hold for later iterations, which proves the claim. 
\end{proof}

We now provide an example for which the PFP performs a general (non-edge) circuit walk.
Consider the networks in Figure \ref{fig:pfp_gen_walk}; the solid arcs $(i,j)$ are labeled with capacity values $r_{ij}$ while the dotted slack arcs $s_i^-$ for each node $i$ are labeled with their current flow value. Figure \ref{fig:pre-processing} depicts the residual network at the beginning of the main loop for the PFP. At this point, the only active nodes are 1, 2, and 3. The Generic PFP does not specify an order in which these active nodes must be picked, and it is easy to specify a rule that would lead to the following run: one could choose node 3 as the first active node, and send flow $x_{34} = 3$ which corresponds to a circuit step of the form $\bm{g} = 3(\bm{e_{34}}, \bm{0_m}, \bm{e_4}-\bm{e_3})$. In the next two iterations, flow could be sent along arcs $(2,4)$ and $(1,2)$, and we arrive at the residual network depicted in Figure \ref{fig:3_steps}. Note that the associated $\bm{y}=(\bm{x},\bm{s^+},\bm{s^-})$ is not a vertex of (\ref{pseudo_poly_form}). To this end, recall that a vertex must be described through an inclusion-maximal set of active constraints; however, a flow of $1$ could be sent along the path circuit for $(2,4)$ and would create a strict superset of active constraints through $x_{24}= u_{24}$. Thus, the PFP performs a general walk on a face of (\ref{pseudo_poly_form}).

\begin{figure}[H]
     \centering
     \begin{subfigure}[b]{0.495\textwidth}
         \centering
         \subcaptionbox{Residual network after pre-processing stage \label{fig:pre-processing}}{
\begin{tikzpicture}[scale = 0.8]
        \tikzset{vertex/.style = {shape=circle,draw,minimum size=2em}}
        \tikzset{edge/.style = {-triangle 90,fill=black}}
        \tikzset{edgedd/.style = {dashed,-triangle 90,fill=black}}
        \node[vertex] (s) at (-2,0) {$s$};
        \node[vertex] (1) at (-0.75,2) {1};
        \node[vertex] (2) at (1,2) {2};
        \node[vertex] (3) at (-0.25,-1) {3};
        \node[vertex] (4) at (1.5,0) {4};
        \node[vertex] (t) at (3.25,0) {$t$};
        
        \draw[-latex,thick] (1)--(s) node[midway, xshift=0mm, yshift = 0mm, fill = white] {$3$};
        \draw[-latex,thick] (2)--(s) node[midway, xshift=0mm, yshift = 0mm, fill = white] {$3$};
        \draw[-latex,thick] (3)--(s) node[midway, xshift=0mm, yshift = 0mm, fill = white]{$3$};
        \draw[-latex,thick] (1)--(2) node[midway, xshift=-1mm, yshift = 0mm, fill = white] {$2$};
        \draw[-latex,thick] (2)--(4) node[midway, xshift=0mm, yshift = 0mm, fill = white] {$4$};
        \draw[-latex,thick] (3)--(4) node[midway, xshift=-1mm, yshift = 0mm, fill = white] {$4$};
        \draw[-latex,thick] (4)--(t) node[midway, xshift=0mm, yshift = 0mm, fill = white] {$9$};

        \draw[-latex,thick,dotted] (-2,2)--(s) node[midway, xshift=0mm, yshift = 0mm, fill = white] {$9$};
        \draw[-latex,thick,dotted] (1)--(-0.75,3.5) node[midway, xshift=0mm, yshift = -1mm, fill = white] {$3$};
        \draw[-latex,thick,dotted] (2)--(1,3.5) node[midway, xshift=0mm, yshift = -1mm, fill = white] {$3$};
        \draw[-latex,thick,dotted] (3)--(-0.25,-2.5) node[midway, xshift=0mm, yshift = 1mm, fill = white] {$3$};
        \draw[-latex,thick,dotted] (4)--(3,-2) node[midway, xshift=0mm, yshift = 0mm, fill = white] {$0$};
        \draw[-latex,thick,dotted] (t)--(3.25,2) node[midway, xshift=0mm, yshift = 0mm, fill = white] {$0$};

    \end{tikzpicture}}
     \end{subfigure}
     \begin{subfigure}[b]{0.495\textwidth}
         \centering
         \subcaptionbox{Residual network after three iterations \label{fig:3_steps}}{
 \begin{tikzpicture}[scale = 0.8]
        \tikzset{vertex/.style = {shape=circle,draw,minimum size=2em}}
        \tikzset{edge/.style = {-triangle 90,fill=black}}
        \tikzset{edgedd/.style = {dashed,-triangle 90,fill=black}}
        \node[vertex] (s) at (-2,0) {$s$};
        \node[vertex] (1) at (-0.75,2) {1};
        \node[vertex] (2) at (1,2) {2};
        \node[vertex] (3) at (-0.25,-1) {3};
        \node[vertex] (4) at (1.5,0) {4};
        \node[vertex] (t) at (3.25,0) {$t$};
        
        \draw[-latex,thick] (1)--(s) node[midway, xshift=0mm, yshift = 0mm, fill = white] {$3$};
        \draw[-latex,thick] (2)--(s) node[midway, xshift=0mm, yshift = 0mm, fill = white] {$3$};
        \draw[-latex,thick] (3)--(s) node[midway, xshift=0mm, yshift = 0mm, fill = white]{$3$};
        \draw[-latex,thick] (2)--(1) node[midway, xshift=1mm, yshift = 0mm, fill = white] {$2$};
        \draw[-latex,thick] (2)--(4) node[midway, xshift=0mm, yshift = 0mm, fill = white] {$1$};
        \draw[-latex,thick] (4) to [bend right = 25] node[midway, xshift=0mm, yshift = 0mm, fill = white] {$3$} (2) ;
        \draw[-latex,thick] (3)--(4) node[midway, xshift=-1mm, yshift = 0mm, fill = white] {$1$};
        \draw[-latex,thick] (4) to [bend right = 25] node[midway, xshift=0mm, yshift = 0mm, fill = white] {$3$} (3) ;
        \draw[-latex,thick] (4)--(t) node[midway, xshift=0mm, yshift = 0mm, fill = white] {$9$};

        \draw[-latex,thick,dotted] (-2,2)--(s) node[midway, xshift=0mm, yshift = 0mm, fill = white] {$9$};
        \draw[-latex,thick,dotted] (1)--(-0.75,3.5) node[midway, xshift=0mm, yshift = -1mm, fill = white] {$1$};
        \draw[-latex,thick,dotted] (2)--(1,3.5) node[midway, xshift=0mm, yshift = -1mm, fill = white] {$2$};
        \draw[-latex,thick,dotted] (3)--(-0.25,-2.5) node[midway, xshift=0mm, yshift = 1mm, fill = white] {$0$};
        \draw[-latex,thick,dotted] (4)--(3,-2) node[midway, xshift=0mm, yshift = 0mm, fill = white] {$6$};
        \draw[-latex,thick,dotted] (t)--(3.25,2) node[midway, xshift=0mm, yshift = 0mm, fill = white] {$0$};
    \end{tikzpicture}}
     \end{subfigure}
        \caption{An example for a run of the PFP where a general walk is performed}
        \label{fig:pfp_gen_walk}
\end{figure}
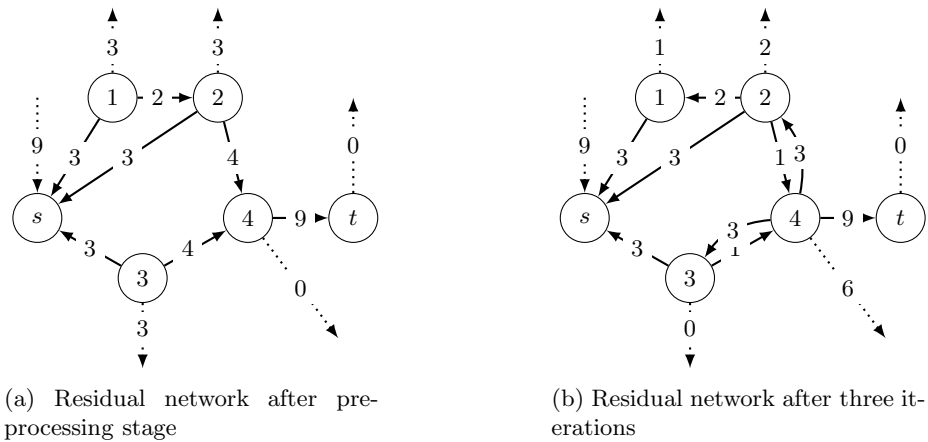

We are interested in a pivot rule which can replicate this walk, but this remains an open question. For the algorithms in Section \ref{sec:alg_of_int} that use dual information, we were able to translate this information to a primal setting through the relaxation from flows to pseudoflows and a primal interpretation in the pseudoflow network. For the PFP, these same techniques do not carry over. We see the development of tools for an interpretation of combinatorial algorithms as primal-dual circuit augmentation for certain pivot rules as the next natural step in this direction of research. This may not only allow a replication of the PFP, but also of other network flows algorithms such as the out-of-kilter algorithm, and algorithms for more general classes of problems.


\bibliographystyle{plain}

\end{document}